\documentclass[twoside]{aiml}

\usepackage{aimlmacro}

\usepackage{graphicx}
\usepackage{amsmath}
\usepackage{amssymb}

\usepackage{tikz}
\usepackage{pgf} 
\usetikzlibrary{patterns,automata,arrows,shapes,snakes,topaths,trees,backgrounds,positioning,through,calc}
\usepackage{comment}

\usepackage{thmtools,thm-restate}

\newcommand{\lat}{\mathfrak{L}}
\newcommand{\comp}{\vartriangleleft}
\newcommand{\compflip}{\vartriangleright}

\def\lastname{Holliday}

\begin{document}

\begin{frontmatter}
  \title{Modal logic, fundamentally}
  \author{Wesley H. Holliday}\vspace{-.05in}
  \address{University of California, Berkeley}\vspace{-.05in}

      \subtitle{{\footnotesize Forthcoming in \textit{Advances in Modal Logic}, Vol.~15, 2024.}}

  \begin{abstract}
  
  Non-classical generalizations of classical modal logic have been developed in the contexts of constructive mathematics and natural language semantics. In this paper, we discuss a general approach to the semantics of non-classical modal logics via algebraic representation theorems. We begin with complete lattices $L$ equipped with an antitone operation $\neg$ sending $1$ to $0$, a completely multiplicative operation $\Box$, and a completely additive operation $\Diamond$. Such lattice expansions can be represented by means of a set  $X$ together with binary relations $\vartriangleleft$, $R$, and $Q$, satisfying some first-order conditions, used to represent $(L,\neg)$, $\Box$, and $\Diamond$, respectively. Indeed, any lattice $L$ equipped with such a $\neg$, a multiplicative $\Box$, and an additive $\Diamond$ embeds into the lattice of propositions of a frame $(X,\vartriangleleft,R,Q)$. Building on our recent study of \textit{fundamental logic}, we focus on the case where $\neg$ is dually self-adjoint ($a\leq \neg b$ implies $b\leq\neg a$) and $\Diamond \neg a\leq\neg\Box a$. In this case, the representations can be constrained so that $R=Q$, i.e., we need only add a single relation to $(X,\vartriangleleft)$ to represent both $\Box$ and $\Diamond$. Using these results, we prove that a system of fundamental modal logic is sound and complete with respect to an elementary class of bi-relational structures $(X,\vartriangleleft, R)$.\end{abstract}

 \end{frontmatter}

\section{Introduction}\label{Intro}
In classical modal logic, necessity and possibility are duals in the sense that $\Box a=\neg\Diamond \neg a$ and $\Diamond a =\neg\Box\neg a$, putting the point algebraically, so typically just one is taken as primitive and the other is treated as defined. The same is true in certain non-classical modal logics, such as the \textit{epistemic orthologic} of \cite{Holliday-Mandelkern2022}. However, in standard treatments of \textit{intuitionistic} modal logic \cite{FischerServi1977,BozicDosen1984,Wijesekera1990}, $\neg\Box a$ does not entail $\Diamond \neg a$, just as in intuitionistic predicate logic, $\neg\forall xP(x)$ does not entail $\exists x\neg P(x)$. In this setting, both $\Box$ and $\Diamond$ must be taken as primitive. Thus, a general approach to non-classical modal logic should do the same. In this paper, building on \cite{Holliday2022,Holliday2023}, we study an approach to the semantics of non-classical modal logics incorporating Plo\v{s}\v{c}ica's \cite{Ploscica1995} approach to the representation of lattices, Birkhoff's \cite{Birkhoff1940} approach to the representation of negation, and the J\'{o}nsson-Tarski \cite{Jonsson1952a} approach to the representation of modal operations; a similar approach without negation was earlier investigated in \cite{Conradie2019}.  Here we add to our treatment of $\neg$ and $\Box$ in \cite{Holliday2022,Holliday2023} a new representation of $\Diamond$.

Our motivation for doing so comes from our recent study of \textit{fundamental logic}, a sublogic of both intuitionistic logic \cite{Heyting1930} and orthologic \cite{Goldblatt1974}. Fundamental propositional logic is defined in \cite{Holliday2023} in terms of a Fitch-style natural deduction system containing only introduction and elimination rules for the logical connectives  $\wedge$, $\vee$, and $\neg$. Thus, unlike Fitch's \cite{Fitch1952,Fitch1966} proof system for classical logic, the Fitch-style proof system for fundamental logic does not contain the rule of Reductio Ad Absurdum (if assuming $\neg\varphi$ leads to a contradiction, conclude $\varphi$) or the rule of Reiteration (which allows pulling previously derived formulas into a subproof). Motivations for dropping Reductio Ad Absurdum include the usual constructive ones, while motivations for dropping Reiteration come from applications to natural language \cite{Holliday-Mandelkern2022}, as well as quantum logic \cite{Chiara2002}. Adding Reductio Ad Absurdum to fundamental logic yields orthologic, while adding Reiteration yields intuitionistic logic in the $\{\wedge,\vee,\neg\}$-fragment. Adding both Reductio and Reiteration gives us back classical logic.

In light of arguments that reasoning with epistemic modals motivates moving from classical logic to orthologic \cite{Holliday-Mandelkern2022} and arguments that reasoning with vague predicates motivates moving from classical to intuitionistic modal logic \cite{Bobzien2020}, it is natural to inquire into extending fundamental logic with modalities. Doing so calls for taking both $\Box$ and $\Diamond$ as primitive, as in intuitionistic modal logic. To accomplish this, we can use \textit{two} accessibility relations, say $R$ for $\Box$ and $Q$ for $\Diamond$. However, we shall see that in the setting of fundamental logic (in which $\neg$ is dually self-adjoint, i.e., $a\leq\neg b$ implies $b\leq\neg a$), just one natural assumption about the interaction of possibility, necessity, and negation, namely that $\Diamond\neg a\leq\neg\Box a$, enables us to use a \textit{single} accessibility relation for both $\Box$ and $\Diamond$. Thus, we will give a simple semantics for fundamental modal logic using bi-relational structures $(X,\comp, R)$ in which $(X,\comp)$ determines a lattice of propositions with negation, and $R$ determines both $\Box$ and $\Diamond$ on the lattice.

In \S~\ref{Background}, we review the background of this project: the system of fundamental logic (\S~\ref{FL}), its algebraic semantics (\S~\ref{AlgSem}), and its relational semantics (\S~\ref{RelSem}). In \S~\ref{Modalities}, we add modalities to the picture and present two representation theorems for lattices with weak negations and independent $\Box$ and $\Diamond$ operations. At this stage, no interaction axioms between $\neg$, $\Box$, and $\Diamond$ are assumed. We study such interactions in \S~\ref{Interactions}, which leads in \S~\ref{UnificationSection} to the appealing simplification mentioned above: in the setting of fundamental logic, assuming $\Diamond\neg a\leq\neg\Box a$ allows us to unify the two accessibility relations for $\Box$ and $\Diamond$. Then from a representation theorem in \S~\ref{UnificationSection}, we obtain the completeness of fundamental modal logic with respect to our bi-relational semantics in \S~\ref{FML}. We conclude in~\S~\ref{Conclusion}.

\section{Background}\label{Background}

\subsection{Fundamental logic}\label{FL}
As noted in \S~\ref{Intro}, the primary definition of fundamental logic in \cite{Holliday2023} is in terms of a Fitch-style proof system with  introduction and elimination rules for $\wedge,\vee,\neg$. For the sake of space, here we will use a secondary but equivalent definition of fundamental logic from \cite{Holliday2023} as a certain \textit{binary logic} in the sense of \cite{Goldblatt1974}. 

Let $\mathcal{L}$ be the language of propositional logic generated from a countably infinite set $\mathsf{Prop}$ of propositional variables by $\wedge$, $\vee$,~and~$\neg$.
\begin{definition}\label{LogicDef} An \textit{intro-elim logic} is a binary relation $\vdash\,\subseteq\mathcal{L}\times\mathcal{L}$ such that for all $\varphi,\psi,\chi\in\mathcal{L}$:
\begin{center}
\begin{tabular}{ll}
1. $\varphi\vdash\varphi$ & 8. if $\varphi\vdash\psi$ and $\psi\vdash\chi$, then $\varphi\vdash\chi$  \\
2. $\varphi\wedge\psi\vdash\varphi$ & \\
3.  $\varphi\wedge\psi\vdash\psi$ & 9. if $\varphi\vdash \psi$ and $\varphi\vdash\chi$, then $\varphi\vdash \psi\wedge\chi$\\
4. $\varphi\vdash \varphi\vee\psi$ & \\
5. $\varphi\vdash \psi\vee\varphi$ & 10. if $\varphi\vdash\chi$ and $\psi\vdash\chi$, then $\varphi\vee\psi \vdash \chi$ \\
6. $\varphi\vdash \neg\neg\varphi$ &  \\
7. $\varphi\wedge\neg\varphi\vdash\psi$ \qquad\qquad\qquad& 11. if  $\varphi\vdash\psi$, then $\neg\psi\vdash\neg\varphi$. 
\end{tabular}
\end{center}
We call the smallest intro-elim logic \textit{fundamental logic}, denoted $\vdash_\mathsf{F}$.
\end{definition}

\textit{Orthologic} \cite{Goldblatt1974}, denoted $\vdash_\mathsf{O}$, is obtained from fundamental logic by adding \textit{double negation elimination}: $\neg\neg\varphi\vdash\varphi$.  \textit{Intuitionistic logic} in the $\{\wedge,\vee,\neg\}$-fragment \cite{Rebagliato1993} is obtained from fundamental logic by strengthening Definition \ref{LogicDef}.6/11 to the \textit{psuedocomplementation} rule that if $\varphi\wedge\psi\vdash\varphi\wedge\neg\varphi$, then $\varphi\vdash\neg\psi$,
 and strengthening proof-by-cases in Definition \ref{LogicDef}.10 to proof-by-cases \textit{with side assumptions}: if $\alpha\wedge\varphi\vdash\chi$ and $\alpha\wedge\psi\vdash \chi$, then $\alpha\wedge (\varphi\vee\psi)\vdash\psi$. \textit{Classical logic}, denoted $\vdash_\mathsf{C}$, is obtained by strengthening fundamental logic with double negation elimination and either of the intuitionistic rules just mentioned \cite[Prop.~3.7]{Holliday-Mandelkern2022}. Of course, there are also weaker logics (in their common signature) than fundamental logic (see \cite{Battilotti1999} and Remark 1.2 of \cite{Holliday2023}).

Aguilera and Byd\u{z}ovsk\'y \cite{Aguilera2022} show that fundamental logic can also be presented in terms of a Gentzen-style  sequent calculus where sequents can have at most one formula on the right, as for intuitionistic logic \cite{Gentzen1935}, and at most two formulas altogether, as for orthologic \cite{Monting1981}. By analyzing this sequent calculus, they show that unlike classical and intuitionistic logic, but like orthologic, fundamental logic is decidable in polynomial time.

\begin{theorem}[\cite{Aguilera2022}] It is decidable in polynomial time whether $\varphi\vdash_\mathsf{F}\psi$.
\end{theorem}

Recall the negative translation of classical into intuitionistic logic \cite{Godel1933,Gentzen1936}: 
\begin{center}
\begin{tabular}{ll}
$g(p)=\neg\neg p$ & $g(\varphi\wedge\psi)=(g(\varphi)\wedge g(\psi))$\\
 $g(\neg\varphi)=\neg g(\varphi)$ & $g(\varphi\vee\psi)=g(\neg(\neg\varphi\wedge\neg\psi))$.
\end{tabular}
\end{center}
As shown in \cite{Holliday2023}, this translation is also a full and faithful embedding of orthologic into fundamental logic.

\begin{proposition}[\cite{Holliday2023}] For all $\varphi,\psi\in\mathcal{L}$, we have $\varphi\vdash_\mathsf{O}\psi$ iff $g(\varphi)\vdash_\mathsf{F}g(\psi)$.
\end{proposition}

We can also carry out classical reasoning inside fundamental logic, but given that the problem of checking $\varphi\vdash_\mathsf{C}\psi$ is co-\textsf{NP}-complete and that of checking $\varphi\vdash_\mathsf{F}\psi$ is in \textsf{P}, we cannot hope for a polynomial-time reduction. Yet we can carry out a reduction at the expense of an exponential blowup in formula length. Given a propositional formula $\varphi$, let $\mathsf{Prop}(\varphi)$ be the set of variables occurring in $\varphi$. Given a set $P=\{p_1,\dots,p_n\}$ of propositional variables, we define the set of \textit{state descriptions over $P$}, $sd(P)$, as the set of all conjunctions of the form $\pm_1 p_1\wedge\dots\wedge \pm_n p_n$ where $\pm_i$ is $\neg$ or empty. The following result shows that $\psi$ is classically derivable from $\varphi$ iff $\psi$ is fundamentally derivable from the assumption that ``there is some determinate way that reality is (in the relevant respects) together with $\varphi$.'' Hence classical propositional logic can be seen as obtained from a logical core of fundamental propositional logic by strengthening the premises of arguments with certain metaphysical assumptions. 

\begin{proposition}\label{ClassicalToFundamental} For any  $\varphi,\psi\in\mathcal{L}$, the following are equivalent:
\begin{enumerate}
\item $\varphi\vdash_{\mathsf{C}}\psi$;
\item $\underset{\delta\in sd(\mathsf{Prop}(\varphi)\cup \mathsf{Prop}(\psi))}{\bigvee}{(\delta \wedge \varphi)}\vdash_{\mathsf{F}} \psi$.
\end{enumerate}
\end{proposition}
\noindent We will prove Proposition \ref{ClassicalToFundamental}  using the relational semantics in \S~\ref{RelSem}.

\subsection{Algebraic semantics}\label{AlgSem}

Algebraic semantics for fundamental logic can be given using bounded lattices---crucially not assumed to be distributive---equipped with what \cite{Dzik2006,Dzik2006b,Almeida2009} call a \textit{weak pseudocomplementation}.

\begin{definition}\label{WeakPseudo} A unary operation $\neg$ on a bounded lattice is a \textit{weak pseudocomplementation} if it satisfies:
\begin{enumerate}
\item semicomplementation: $a\wedge\neg a =0$;
\item dual self-adjointness: $a\leq \neg b$ implies $b\leq \neg a$.
\end{enumerate}
\end{definition}

The following easy folklore lemma relates Definitions \ref{WeakPseudo} and \ref{LogicDef}.

\begin{lemma}\label{DSA} $\neg$ is dually self-adjoint iff it is antitone \textnormal{(}$a\leq b$ implies $\neg b\leq \neg a$\textnormal{)} and double inflationary \textnormal{(}$a\leq\neg\neg a$\textnormal{)}.
\end{lemma}

We assume familiarity with how a class $\mathbb{C}$ of lattices with a unary operation $\neg$ provides algebraic semantics for $\mathcal{L}$ and a consequence relation ${\vDash_\mathbb{C}\,\subseteq\mathcal{L}\times\mathcal{L}}$. Standard techniques of algebraic logic then yield the following (see \cite{Holliday2023}).

\begin{proposition} Fundamental logic is sound and complete with respect to the class $\mathbb{WPL}$ of bounded lattices equipped with a weak pseudocomplementation: for all $\varphi,\psi\in\mathcal{L}$, $\varphi\vdash_\mathsf{F}\psi$ iff $\varphi\vDash_\mathbb{WPL}\psi$.
\end{proposition}

\subsection{Relational semantics}\label{RelSem}

A relational semantics for fundamental logic \cite{Holliday2023} can be given using Plo\v{s}\v{c}ica's \cite{Ploscica1995} approach to the representation of lattices, Birkhoff's \cite{Birkhoff1940} approach to the representation of negation, and appropriate additional first-order conditions on the relations. For comparisons with related works \cite{Almeida2009,Chiara2002,Dosen1984,Dosen1986,Dosen1999,Dunn1993,Dunn1996,Dunn1999,Dunn2005,Dzik2006,Dzik2006b,Massas2023,Vakarelov1989,Holliday2021,Zhong2021}, as well as examples, see \cite[\S~4]{Holliday2023}. Here we only quickly summarize the key facts concerning this semantics.

\begin{definition}\label{RelFrame} A \textit{relational frame} is a pair $(X,\comp)$ of a nonempty set $X$ and binary relation $\comp$ on $X$.
\end{definition}
If $x\comp y$, we say that $x$ is \textit{open to} $y$. As explained in Remark 4.2 of \cite{Holliday2023}, this reading of $\comp$ is associated with a four-way distinction between acceptance, non-acceptance, rejection, and acceptance of the negation of a proposition $A$, where propositions are fixpoints of the operation $c_\comp$ in Proposition \ref{RelSemFacts} below:
\begin{itemize}
\item $x$ \textit{accepts} $A$ if $x\in A$;
\item $x$ does not accept $A$ if $x\not\in A$;
\item $x$ \textit{rejects} $A$ if for all $y\compflip x$, $y\not\in A$;
\item $x$ accepts \textit{the negation of $A$} if for all $y\comp x$, $y\not\in A$.
\end{itemize}
Non-acceptance should not entail rejection, since a state may be completely noncommittal about $A$; and rejection should not entail acceptance of the negation, since we would like to accommodate, e.g., intuitionists who reject instances of excluded middle but of course do not accept their negations. Given these distinctions, we can provide more intuition to the notion of ``openness'' intended for $\comp$: $x$ is open to $y$ iff $x$ \textit{does not reject any proposition that $y$ accepts}.

\begin{proposition}\label{RelSemFacts} Let $(X,\comp)$ be a relational frame.
\begin{enumerate}
\item The following operation $c_\comp: \wp (X)\to \wp(X)$ is a closure operator:
\[c_\comp (A)=\{x\in X\mid \forall y\comp x\,\exists z\compflip y: z\in A\} .\]
\item The fixpoints of $c_\comp$, i.e., those $A\subseteq X$ with $c_\comp (A)=A$, form a complete lattice $\lat(X,\comp)$ with meet as intersection and join as \emph{closure of} union:
\[\underset{i\in I}{\bigvee }A_i=\{x\in X\mid \forall y\comp x\,\exists z \compflip y: z\in \underset{i\in I}{\bigcup} A_i\}.\]
\item The operation  $\neg_\comp: \wp(X)\to\wp(X)$ defined by
\[\neg_\comp(A)=\{x\in X\mid \forall y\comp x, y\not\in A\}\]
sends $c_\comp$-fixpoints to $c_\comp$-fixpoints, is antitone with respect to $\subseteq$, and sends the $1$ of $\lat(X,\comp)$, namely $X$, to the $0$ of $\lat(X,\comp)$, namely $c_\comp(\varnothing)$.
\end{enumerate}
\end{proposition}

For fundamental logic, we want $\neg_\comp$ to have additional properties. For the following result, say that an $x\in X$ is \textit{non-absurd} if there is some $y\comp x$. Given $x,z\in X$, say that $z$ \textit{pre-refines} $x$ if for all $w\comp z$, we have $w\comp x$. It follows that for all $c_\comp$-fixpoints $A$, if $x\in A$, then $z\in A$ \cite[Lemma 4.12]{Holliday2023}.

\begin{proposition}[\cite{Holliday2023}, Proposition 4.14.1-2]\label{CorrespondenceProp} For any relational frame $(X,\comp)$, in each of the following pairs, \textnormal{(a)} and \textnormal{(b)} are equivalent: 
\begin{enumerate}
\item\label{CorrespondenceProp1} 
\begin{enumerate}
\item for all $c_\comp$-fixpoints $A$, we have $A\cap \neg_\comp A=0$; 
\item \emph{pseudo-reflexivity}: for all non-absurd $x\in X$, there is a $z\comp x$ that pre-refines $x$.
\end{enumerate}
\item\label{CorrespondenceProp2} \begin{enumerate}
\item for all $c_\comp$-fixpoints $A$, we have $A\subseteq\neg_\comp\neg_\comp A$;
\item \emph{pseudo-symmetry}: for all $x\in X$ and $y\comp x$, there is a $z\comp y$ that pre-refines $x$.
\end{enumerate}
\end{enumerate}
\end{proposition}

The facts above yield the soundness of fundamental logic with respect to relational frames that are pseudo-reflexive and pseudo-symmetric, interpreting $\mathcal{L}$ in the algebras $(\lat(X,\comp),\neg_\comp)$. That is, a \textit{relational model} adds to a relational frame $(X,\comp)$ a valuation $V$ interpreting propositional variables as fixpoints of $c_\comp$. The forcing relation $\Vdash$ between states $x\in X$ and formulas is defined in the obvious way in light of Proposition \ref{RelSemFacts}(ii)-(iii) \cite[Definition~4.19]{Holliday2023}.

 For completeness, we use the following representation theorem.

\begin{theorem}[\cite{Holliday2023}, Theorems 4.24 and 4.30] Any bounded \textnormal{(}resp.~complete\textnormal{)} lattice  $L$ equipped with an antitone operation $\neg$ sending $1$ to $0$ embeds into \textnormal{(}resp.~is isomorphic to\textnormal{)} $(\lat(X,\comp),\neg_\comp)$ for some relational frame~$(X,\comp)$. 

Moreover, if $\neg$ satisfies  $a\leq\neg\neg a$ \textnormal{(}resp.~$a\wedge\neg a=0$\textnormal{)} for all $a\in L$, then $\comp$ may be taken to be  pseudo-symmetric \textnormal{(}resp.~pseudo-reflexive---in fact, reflexive\textnormal{)}.
\end{theorem}

\begin{theorem}[\cite{Holliday2023}]\label{Completeness} Fundamental logic is sound and complete with respect to the class of relational frames that are pseudo-reflexive and pseudo-symmetric.
\end{theorem}

As an example application of Theorem \ref{Completeness}, let us prove Proposition \ref{ClassicalToFundamental}.

\begin{proof} From 2 to 1, $\varphi\vdash_\mathsf{C} \underset{\delta\in sd(\mathsf{Prop}(\varphi)\cup \mathsf{Prop}(\psi))}{\bigvee}{(\delta \wedge \varphi)}\vdash_{\mathsf{F}} \psi$ and hence $\varphi\vdash_\mathsf{C}\psi$.\\

Now suppose 2 does not hold. Then there is $\delta\in sd(\mathsf{Prop}(\varphi)\cup \mathsf{Prop}(\psi))$ such that $\delta\wedge\varphi\nvdash_\mathsf{F}\psi$, for otherwise 2 holds using Definition \ref{LogicDef}.10. Then by Theorem \ref{Completeness}, there is a  pseudo-reflexive and pseudo-symmetric model $\mathcal{M}=(X,\comp,V)$ and $x\in X$ such that $\mathcal{M},x\Vdash\delta\wedge\varphi$ but $\mathcal{M},x\nVdash\psi$. Define a valuation $\pi:\mathsf{Prop}\to \{0,1\}$ by $\pi(p)=1$ if $\mathcal{M},x\Vdash p$ and $0$ otherwise. Let $\tilde{\pi}:\mathcal{L}\to\{0,1\}$ be the usual recursively defined extension of $\pi$ as in classical semantics.

We prove by induction that for any propositional formula $\chi$ with $\mathsf{Prop}(\chi)\subseteq \mathsf{Prop}(\varphi)\cup\mathsf{Prop}(\psi)$, we have:
\begin{enumerate}
\item[(a)] $\tilde{\pi}(\varphi)=1$ iff $\mathcal{M},x\Vdash\chi$; (b) $\tilde{\pi}(\varphi)=0$ iff $\mathcal{M},x\Vdash\neg\chi$.
\end{enumerate}

Suppose $\chi$ is a propositional variable $p$. Then since $\chi\in \mathsf{Prop}(\varphi)\cup\mathsf{Prop}(\psi)$, from $\mathcal{M},x\Vdash\delta$ it follows that either $\mathcal{M},x\Vdash p$, in which case $\tilde{\pi}(p)=1$ by definition of $\pi$, or $\mathcal{M},x\Vdash\neg p$, which implies $\mathcal{M},x\nVdash p$ by the pseudo-reflexivity of $\comp$, so $\tilde{\pi}(p)=0$ by definition of $\pi$. This establishes (a) and (b) for $p$.

Suppose $\chi$ is $\neg\alpha$.  If $\tilde{\pi}(\alpha)=0$, then by the inductive hypothesis, $\mathcal{M},x\Vdash \neg\alpha$, in line with $\tilde{\pi}(\neg\alpha)=1$. On the other hand, if $\tilde{\pi}(\alpha)=1$, then by the inductive hypothesis, $\mathcal{M},x\Vdash \alpha$, which by the pseudo-symmetry of $\comp$ implies $\mathcal{M},x\Vdash \neg\neg\alpha$, in line with $\tilde{\pi}(\neg\alpha)=0$. Thus, (a) and (b) hold for $\neg\alpha$.

Suppose $\chi$ is $\alpha\wedge\beta$. Simply consider the four possible truth assignments to $\alpha,\beta$ by $\tilde{\pi}$ and use the fact that $\mathcal{M},x\Vdash \neg\gamma_i$ implies  $\mathcal{M},x\Vdash \neg (\gamma_1\wedge\gamma_2)$.

Finally, suppose $\chi$ is $\alpha\vee\beta$. Again consider the four possible truth assignments to $\alpha,\beta$ by $\tilde{\pi}$ and use the fact that $\mathcal{M},x\Vdash \neg\alpha\wedge\neg\beta$ implies $\mathcal{M},x\Vdash \neg (\alpha\vee\beta)$. To see this, suppose $\mathcal{M},x\Vdash \neg\alpha\wedge\neg\beta$ and $y\comp x$. For contradiction, suppose $\mathcal{M},y\Vdash \alpha\vee \beta$. Given  $y\comp x$ and the pseudo-symmetry of $\comp$, there is a $z\comp y$ that pre-refines $x$, so $\mathcal{M},z\Vdash \neg\alpha\wedge\neg\beta$. Since $\mathcal{M},y\Vdash \alpha\vee \beta$ and $z\comp y$, there is a $w\compflip z$ with $\mathcal{M},w\Vdash \alpha$ or $\mathcal{M},w\Vdash\beta$. Given  $w\compflip z$  and the pseudo-symmetry of $\comp$, there is a $u\comp z$ that pre-refines $w$, so  $\mathcal{M},u\Vdash \alpha$ or $\mathcal{M},u\Vdash\beta$. But this contradicts $u\comp z$ together with $\mathcal{M},z\Vdash \neg\alpha\wedge\neg\beta$. Thus, we conclude $\mathcal{M},y\nVdash \alpha\vee \beta$, which shows that $\mathcal{M},x\Vdash \neg (\alpha\vee\beta)$.

Now given (a), $\mathcal{M},x\Vdash \varphi$ implies $\tilde{\pi}(\varphi)=1$, and $\mathcal{M},x\nVdash \psi$ implies $\tilde{\pi}(\psi)\neq 1$. Then by the soundness of classical logic with respect to its standard valuation semantics, $\varphi\nvdash_{\mathsf{C}}\psi$, so we are done.\end{proof}

\section{Modalities}\label{Modalities}

\subsection{Algebras and frames}\label{AlgFrames}

To add modalities to our story, let us recall the following standard definitions.

\begin{definition}
A unary operation $f$ on a lattice $L$ is \textit{monotone} if $a\leq b$ implies $f(a)\leq f(b)$. We say that $f$ is \textit{multiplicative} (resp.~\textit{completely multiplicative}) if for any finite (resp.~arbitrary) subset $S$ of elements of $L$ (such that $\bigwedge S$ exists),
\[f( \bigwedge S)=\bigwedge \{f(a)\mid a\in S\}.\]
Dually, $f$  is \textit{additive} (resp.~\textit{completely additive}) if for any finite (resp.~arbitrary) subset $S$ of elements of $L$ (such that $\bigvee S$ exists), 
\[f(\bigvee S)=\bigvee \{f(a)\mid a\in S\}.\]
\end{definition}

As suggested in \S~\ref{Intro}, the first idea for extending the relational semantics of \S~\ref{RelSem} to handle necessity and possibility modals is to add two accessibility relations to $(X,\comp)$; see \cite{Conradie2019} for a similar approach but without negation in the signature. Other related approaches to representing lattices with modalities can be found in, e.g., \cite{Bezhanishvili2024,Conradie2016,Dmitrieva2021,Gehrke2006,Goldblatt2019,Hartonas2018,Hartonas2019,Holliday2015,Holliday2021b,Orlowska2005}.

\begin{definition}\label{ModalFrameDef} A \textit{modal frame} is a triple $(X,\comp,R,Q)$ such that $\comp$, $R$, and $Q$ are binary relations on $X$, and for all $x,y,z\in X$,
\[\mbox{if $x Ry\compflip z$, then $\exists x'\comp x$ $\forall x''\compflip x' $ $\exists y''$: $x''Ry''\compflip z$.}\]
\end{definition}

\begin{figure}[h]
\begin{center}
\begin{tikzpicture}[->,>=stealth',shorten >=1pt,shorten <=1pt, auto,node
distance=2cm,thick,every loop/.style={<-,shorten <=1pt}]
\tikzstyle{every state}=[fill=gray!20,draw=none,text=black]

\node (x) at (0,0) {{$x$}};
\node (y) at (4,0) {{$y$}};
\node (z) at (4,3) {{$z$}};
\node at (5.5,1.5) {{\textit{$\Rightarrow$}}};

\path (x) edge[dashed,->] node {{}} (y);
\path (y) edge[->] node {{}} (z);

\node (x2) at (7,0) {{$x$}};
\node (x'2) at (7,3) {{$x'$}};
\node (x''2) at (8,1) {{$x''$}};
\node (y''2) at (10,1) {{$y''$}};
\node (y2) at (11,0) {{$y$}};
\node (z2) at (11,3) {{$z$}};

\path (x2) edge[dashed,->] node {{}} (y2);
\path (x''2) edge[dashed,->] node[above] {{$\exists$}} (y''2);
\path (y2) edge[->] node {{}} (z2);
\path (x2) edge[->] node {{$\exists$}} (x'2);
\path (x''2) edge[->] node[right] {{$\forall$}} (x'2);
\path (y''2) edge[->] node {{}} (z2);

\end{tikzpicture}
\end{center}
\caption{Illustration of the modal frame condition in Definition \ref{ModalFrameDef}. A solid line from $w$ to $v$ indicates $w\compflip v$, and a dashed line from $w$ to $v$ indicates $wRv$.}
\end{figure}
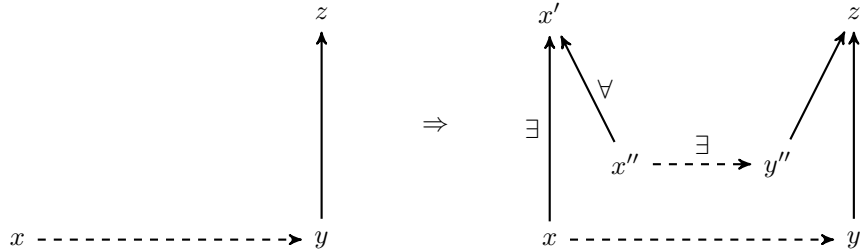

From $R$ we define a necessity modality $\Box_R$ as usual. However, from $Q$ we will define our possibility modality  $\Diamond_Q$ using a more intricate quantificational pattern, as shown in Fig.~\ref{DiamondDefFig} below. For any $S\subseteq X^2$, let $S(x)=\{y\in X\mid xSy\}$.

\begin{proposition}\label{BoxDiamond} Given a modal frame $\mathcal{F}=(X,\comp,R,Q)$, define operations $\Box_R$ and $\Diamond_Q$ on the lattice of $c_\comp$-fixpoints of $\mathcal{F}$ as follows:
\begin{eqnarray*}
\Box_RA & = & \{x\in X\mid R(x)\subseteq A\}; \\
\Diamond_QA & = & \{x\in X\mid \forall x'\comp x \;\exists y'\in Q(x')\,\exists y\compflip y'\colon y\in A \}.
\end{eqnarray*}
Then $\Box_R$ and $\Diamond_Q$ send $c_\comp$-fixpoints to $c_\comp$-fixpoints,  $\Box_R$ is completely multiplicative, and $\Diamond_Q$ is monotone.\end{proposition}
\begin{proof} First, we  show that $\Box_R A$ is a $c_\comp$-fixpoint for any $c_\comp$-fixpoint $A$. Equivalently, we show that
if $ x\in X\setminus \Box_R A$, then  $\exists x'\comp x\,\forall x''\compflip x' \; x''\not\in \Box_R A$.  Suppose $x\not\in \Box_R A$, so there is some $y$ such that  $xRy\not\in A$. Then since $A$ is a {$c_\comp$-fixpoint}, there is a ${z\comp y}$ such that ($\star$) for all $z'\compflip z$, we have $z'\not\in A$. Since $xRy\compflip z$, by the modal frame condition we have $\exists x'\comp x\,\forall x''\compflip x' \,\exists y'':\, x'' R y'' \compflip z$. Now $z\comp y''$ implies $y''\not\in A$ by ($\star$), which with $x''Ry''$ implies $x''\not\in\Box_R A$. 

Next, we show that $\Diamond_QA$ is a $c_\comp$-fixpoint for any $c_\comp$-fixpoint $A$. Suppose $x\not\in \Diamond_QA$, so  $\exists x'\comp x \;\forall y'\in Q(x')\;\forall y\compflip y':y\not\in A$. Then clearly there is no $x''\compflip x'$ with $x''\in \Diamond_Q A$. Hence $\exists x'\comp x$ $\forall x''\compflip x'$ $x''\not\in\Diamond_Q A$, as desired.

That $\Box_R$ is completely additive and $\Diamond_Q$ monotone is obvious from the definitions.\end{proof}

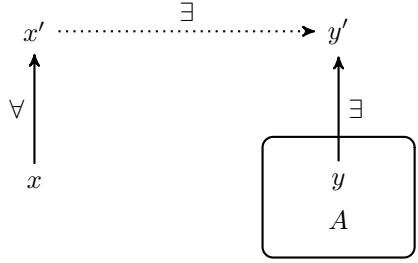
\begin{figure}[h]
\begin{center}
\begin{tikzpicture}[->,>=stealth',shorten >=1pt,shorten <=1pt, auto,node
distance=2cm,thick,every loop/.style={<-,shorten <=1pt}]
\tikzstyle{every state}=[fill=gray!20,draw=none,text=black]

\node (x) at (0,0) {{$x$}};
\node (x') at (0,2) {{$x'$}};
\node (y') at (4,2) {{$y'$}};
\node (y) at (4,0) {{$y$}};

\path (x) edge[->] node {{$\forall$}} (x');
\path (x') edge[dotted,->] node {{$\exists$}} (y');
\path (y) edge[->] node[right] {{$\exists$}} (y');

\path[-, draw=black, opacity=0.5, thick, rounded corners] (5, .5) -- (5, -1) --  (3.95, -1) -- (3, -1) -- (3, .6)  -- (5, .6) -- (5, .5) ;

\node (A) at (4,-.5) {{$A$}}; 

\end{tikzpicture}
\end{center}
\caption{Illustration of the condition for $x\in \Diamond_QA$ from Proposition \ref{BoxDiamond}. The dotted line from $x'$ to $y'$ indicates $x'Qy'$.}\label{DiamondDefFig}
\end{figure}

The definition of $\Diamond_Q$ can be understood intuitively as follows, using notions from \S~\ref{RelSem}. First, let $\Box_Q$ be the necessity modality defined from $Q$ in the usual way, so $\Box_Q A=\{x\in X\mid Q(x)\subseteq A\}$. (So when we take $Q=R$ in \S~\ref{UnificationSection}, this is just $\Box_R$.)  Then $x\in \Diamond_Q A$ in effect means that according to $x$, \[\mbox{\textit{it's not the case that $A$ is necessarily \textnormal{(}relative to $Q$\textnormal{)} rejected}}.\]
For given our interpretation of `not', the displayed condition means there is some $x'\comp x$ that does not accept  that \textit{$A$ is necessarily \textnormal{(}relative to $Q$\textnormal{)} rejected}, which in turn means there is some $y'$ that is $Q$-accessible from $x'$ and does not reject $A$, which in turns means there is a some $y\compflip y'$ with $y\in A$.

\begin{remark}Without further conditions, $\Diamond_Q$ is not guaranteed to be additive. But this is a feature, rather than a bug, of the above approach to possibility, since there are contexts in which additivity is not desired for $\Diamond$. For example, Wijesekera \cite{Wijesekera1990} intentionally designed his system of intuitionistic modal logic so that $\Diamond$ does not distributive over $\vee$, since this is not wanted for some applications of intuitionistic modal logic in computer science. For another example, we recall Kenny's \cite{Kenny1976} argument that the \textit{ability} modality does not distribute over $\vee$: you may be able to ensure that your dart hits the top half of the dart board or your dart hits the bottom half of the board; it does not follow that you are able to ensure that your dart hits the top half or that you are able to ensure that your dart hits the bottom half, since that may be beyond your~skill.\end{remark}

When we want $\Diamond_Q$ to be completely additive, as we now do, we simply impose a condition analogous to that of Definition \ref{ModalFrameDef} for $Q$ but with $\comp$ flipped.

\begin{definition}\label{AddDef} A modal frame $(X,\comp,R,Q)$ is \textit{additive} if for all $x,y,z\in X$,
\[\mbox{if $x Qy\comp z$, then $\exists x'\compflip x$ $\forall x''\comp x'$ $\exists y''$: $x''Qy''\comp z$}.\] 
\end{definition}

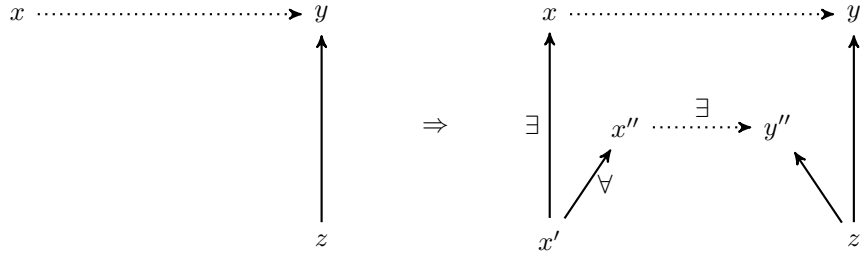
\begin{figure}[h]
\begin{center}
\begin{tikzpicture}[->,>=stealth',shorten >=1pt,shorten <=1pt, auto,node
distance=2cm,thick,every loop/.style={<-,shorten <=1pt}]
\tikzstyle{every state}=[fill=gray!20,draw=none,text=black]

\node (x) at (0,3) {{$x$}};
\node (y) at (4,3) {{$y$}};
\node (z) at (4,0) {{$z$}};
\node at (5.5,1.5) {{\textit{$\Rightarrow$}}};

\path (x) edge[dotted,->] node {{}} (y);
\path (z) edge[->] node {{}} (y);

\node (y2) at (7,3) {{$x$}};
\node (x'2) at (7,0) {{$x'$}};
\node (y'2) at (8,1.5) {{$x''$}};
\node (z'2) at (10,1.5) {{$y''$}};

\node (z2) at (11,3) {{$y$}};
\node (w2) at (11,0) {{$z$}};

\path (x'2) edge[->] node[left] {{$\exists$}} (y2);
\path (x'2) edge[->] node[right] {{$\forall$}} (y'2);
\path (y2) edge[dotted,->] node {{}} (z2);
\path (w2) edge[->] node {{}} (z2);
\path (w2) edge[->] node {{}} (z'2);
\path (y'2) edge[dotted,->] node[above] {{$\exists$}} (z'2);

\end{tikzpicture}
\end{center}
\caption{Illustration of the additivity condition in Definition \ref{AddDef}.}
\end{figure}

\begin{proposition}\label{AddLemm} If $(X,\comp,R,Q)$  is an additive modal frame, then the operation $\Diamond_Q$ is completely additive.
\end{proposition}
\begin{proof} Suppose $w\in \Diamond \bigvee \{A_i\mid i\in I\}$. Toward showing $w\in \bigvee \{\Diamond A_i\mid i\in I\}$, consider some $x\comp w$. Then since $w\in  \Diamond \bigvee \{A_i\mid i\in I\}$, there are $y,u$ such that $xQy\comp u\in \bigvee \{A_i\mid i\in I\}$. Since $u\in \bigvee \{A_i\mid i\in I\}$ and $y\comp u$, it  follows that there is some $z\compflip y$ such that $z\in A_i$ for some $i\in I$. Then picking $x'$ as in Definition \ref{AddDef}, we have $x'\in\Diamond A_i$. Thus, for every $x\comp w$, there is an $x'\compflip x$ and $i\in I$ such that $x'\in \Diamond A_i$, which shows that $w\in  \bigvee \{\Diamond A_i\mid i\in I\}$.\end{proof}

In summary, for $\Box_R$ we impose an interaction condition on $R$ and $\comp$ to ensure that $\Box_R$ send $c_\comp$-fixpoints to $c_\comp$-fixpoints, but no interaction condition is required for $\Box_R$ to be completely multiplicative. By contrast, for $\Diamond_Q$ no interaction condition on $Q$ and $\comp$ is required for $\Diamond_Q$ to send $c_\comp$-fixpoints to $c_\comp$-fixpoints, but we impose an interaction condition when we want $\Diamond_Q$ to be completely additive.

\subsection{Representation}

The representational power of the relational frames from the previous subsection is shown by the following result.

\begin{theorem}\label{CompModalRep} Let $L$ be a complete lattice equipped with $\neg$, $\Box$, and $\Diamond$ where
\begin{itemize}
\item $\neg$ is an antitone unary operation on $L$ with $\neg 1=0$,
\item $\Box$ is a completely multiplicative unary operation on $L$,  and 
\item $\Diamond$ is a completely additive unary operation on $L$. 
\end{itemize}
Then define:
\begin{itemize}
\item $X = \{(a,b)\mid a,b\in L,\neg a\leq b\}$, and for $x=(a,b)\in X$, $x_0=a$, and $x_1=b$;
\item $x\comp y$ iff $x_1\not\geq y_0$;
\item $xRy$ iff  for all $a\in L$, $x_0\leq \Box a\Rightarrow y_0\leq a$;
\item $xQy$ iff for all $a\in L$, $\Diamond a\leq x_1\Rightarrow  a\leq y_1$.
\end{itemize}
 Then $(X,\comp, R,Q)$ is an additive modal frame, and $(L,\neg, \Box,\Diamond)$ is isomorphic to ${(\lat(X,\comp),\neg_\comp, \Box_R,\Diamond_Q)}$.
\end{theorem}

To prove Theorem \ref{CompModalRep}, we make use of the following definition and proposition from  \cite{Holliday2022,Holliday2023}.

\begin{definition}\label{Good} Let $L$ be a lattice and $P$ a set of pairs of elements of $L$. Define a binary relation $\comp$ on $P$ by 
$(a,b)\comp (c,d)$ if $c\not\leq b$. Then we say $P$ is \textit{separating} if for all $a,b\in L$:
\begin{enumerate}
\item\label{Good2} if $a\not\leq b$, then there is a $(c,d)\in P$ with $c\leq a$ and $c\not\leq b$;
\item\label{Good3} for all $(c,d)\in P$, if $c\not\leq b$, then there is a $(c',d')\comp (c,d)$ such that for all $(c'',d'')\compflip (c',d')$, we have $c''\not\leq b$.
\end{enumerate}
\end{definition}

\begin{proposition}\label{CompRep} Let $L$ be a lattice and $P$ a separating set of pairs of elements of $L$. For $a\in L$, define $ f(a)= \{(x,y) \in P\mid x\leq a\}$. Then:
\begin{enumerate}
\item\label{CompRep1} $f$ is a complete embedding of $L$ into $\lat(P,\comp)$;
\item\label{CompRep2} if $L$ is complete, then $f$ is an isomorphism from $L$ to $\lat(P,\comp)$.
\end{enumerate}
\end{proposition}
\noindent In (i), $\mathfrak{L}(P,\comp)$ is the MacNeille completion of $L$ (see  \cite[Thm.~2.2]{Gehrke2005}). For a proof of Proposition \ref{CompRep}, see \cite[Proposition 4.23]{Holliday2023}.  We now prove Theorem~\ref{CompModalRep}.

\begin{proof} We first prove a preliminary lemma that we will use repeatedly. For any $x\in X$, let
\begin{eqnarray*}
\rho(x)&=& (\bigwedge\{b\mid x_0\leq\Box b\},   \neg \bigwedge\{b\mid x_0\leq\Box b\})\\
\sigma(x) &=&(1, \bigvee\{b\mid \Diamond b\leq x_1\}),
\end{eqnarray*}
so $\rho(x),\sigma(x)\in X$. Then obviously (a) $xR\rho(x)$ and  (b) $xQ\sigma(x)$. In addition:
\begin{itemize}
\item[(c)] if $x_0\not\leq \Box a$, then $\rho(x)_0\not\leq a$. Contrapositively,
\begin{eqnarray*}
&&\bigwedge\{b\mid x_0\leq\Box b\}\leq a\\
&\Rightarrow& \Box\bigwedge\{b\mid x_0\leq\Box b\}\leq \Box a\quad\mbox{by monotonicity of }\Box\\
&\Rightarrow& \bigwedge\{\Box b\mid x_0\leq\Box b\}\leq \Box a \quad\mbox{by complete multiplicativity of }\Box \\
&\Rightarrow& x_0\leq \Box a.
\end{eqnarray*}
\item[(d)] if $\Diamond a\not\leq x_1$, then  $a\not\leq \sigma(x)_1$. Contrapositively,
\begin{eqnarray*}
&&a\leq  \bigvee\{b\mid \Diamond b\leq x_1\} \\
&\Rightarrow& \Diamond a\leq \Diamond \bigvee\{b\mid \Diamond b\leq x_1\}\quad\mbox{by the monotonicity of }\Diamond\\
&\Rightarrow& \Diamond a\leq  \bigvee\{\Diamond b\mid \Diamond b\leq x_1\}\quad\mbox{by the complete additivity of }\Diamond\\
&\Rightarrow& \Diamond a\leq   x_1.
\end{eqnarray*}
\end{itemize}

Now we show that $(X,\comp, R,Q)$ is a modal frame as in Definition \ref{ModalFrameDef}:
\[\mbox{if $x Ry\compflip z$, then $\exists x'\comp x$ $\forall x''\compflip x' $ $\exists y''$: $x''Ry''\compflip z$.}\]
Suppose $x Ry\compflip z$. Then $x_0\not\leq \Box z_1$, for otherwise $xRy$ implies $y_0\leq z_1$, contradicting $y\compflip z$.  Now let $x'=(1, \Box z_1)$, so $x'\comp x$. Consider any $x''\compflip x'$, so $x''_0\not\leq x'_1=\Box z_1$. Let $y''=\rho(x'')$, so $x''Ry''$ by (a). Then $x''_0\not\leq \Box z_1$ implies $y''_0\not\leq z_1$ by (c), so $y''\compflip z$, which establishes the modal frame condition.

Next we show that  $(X,\comp, R,Q)$  is additive as in Definition \ref{AddDef}:
\[\mbox{if $x Qy\comp z$, then $\exists x'\compflip x$ $\forall x''\comp x'$ $\exists y''$: $x''Qy''\comp z$}.\] 
Suppose $x Qy\comp z$. Since $y\comp z$, we have $z_0\not\leq y_1$, which with $xQy$ implies $\Diamond z_0\not\leq x_1$. Then where $x'=(\Diamond z_0,\neg \Diamond z_0)$, we have $x'\compflip x$. Now consider any $x''\comp x'$, so $\Diamond z_0\not\leq x_1''$.  Let $y''=\sigma(x'')$, so $x''Qy''$ by (b). Then  $\Diamond z_0\not\leq x_1''$ implies $z_0\not\leq y''_1$ by (d), so $y''\comp z$, which shows that the frame is additive.

Now we prove that $(L,\neg, \Box,\Diamond)$ is isomorphic to ${(\lat(X,\comp),\neg_\comp, \Box_R,\Diamond_Q)}$. First, we claim that $X$ is separating as in Definition \ref{Good}.  For part (\ref{Good2}) of Definition \ref{Good}, take $(c,d)=(a,\neg a)$. For (\ref{Good3}), suppose $(c,d)\in X$ and $c\not\leq b$. Let $(c',d')=(1,b)$. Since $b\neq 1$ and $\neg 1=0\leq b$, $(1,b)\in X$, and since $c\not\leq b$, $(c',d')\comp (c,d)$. Now consider any $(c'',d'')\in X$ with $(c',d')\comp (c'',d'')$. Then $c''\not\leq d'=b$, so (\ref{Good3}) holds. Thus, by Proposition \ref{CompRep}, the $f$ defined there is an isomorphism from $L$ to $\mathfrak{L}(P,\comp)$. Next, we show that $f$ preserves $\neg$, $\Box$, and $\Diamond$.

To show $ f(\neg a)=\neg_\comp f(a)$, first suppose $(x,y)\in  f(\neg a)$, so $x\leq\neg a$, and $(x',y')\comp (x,y)$. If $x'\leq a$, then $\neg a\leq\neg x'$, which with $x\leq\neg a$ implies $x\leq \neg x'$, which with $\neg x'\leq y'$ implies $x\leq y'$, contradicting $(x',y')\comp (x,y)$. Thus, we have $x'\not\leq a$, so $(x',y')\not\in  f(a)$. Hence $(x,y)\in \neg_\comp  f(a)$. Conversely, let $(x,y)\in X\setminus f(\neg a)$, so $x\not\leq \neg a$. Then $(a,\neg a)\comp (x,y)$, so $(x,y)\not\in \neg_\comp f(a)$.

To show $f(\Box b)=\Box_Rf(b)$, first suppose $x\in f(\Box a)$, so $x_0\leq \Box a$. Then $xRy$ implies $y_0\leq a$ and hence $y_0\in f(a)$. Thus, $x\in\Box_Rf(a)$. Conversely, suppose $x\not\in f(\Box a)$, so $x_0\not\leq \Box a$. Let $y=\rho(x)$, 
  so $xRy$ by (a). Then $x_0\not\leq \Box a$ implies $y_0\not\leq a$ by (c), so $y\not\in f(a)$ and hence $x\not\in \Box_R f(a)$. 

To show $f(\Diamond a)=\Diamond f(a)$, first suppose $x\in f(\Diamond a)$, so $x_0\leq \Diamond a$. Further suppose $x'\comp x$, so $x_0\not\leq x'_1$ and hence $\Diamond a\not\leq x'_1$. Let  $y'=\sigma(x')$, so $x'Qy'$ by (b). The $\Diamond a\not\leq x'_1$ implies $a\not\leq y'_1$ by (d), so $y'\comp (a,\neg a)$. Since $(a,\neg a)\in f(a)$, this shows that $x\in \Diamond f(a)$.

Conversely, suppose $x\not\in f(\Diamond a)$, so $x_0\not\leq \Diamond a$. Let $x'=(1,\Diamond a)$, so $x'\comp x$. Now consider any $y'$ such that $x'Qy'$, which with $\Diamond a\leq x'_1$ implies $a\leq y'_1$. Then for any $y\compflip y'$, we have $y_0\not\leq y'_1$ and hence $y_0\not\leq a$, so $y\not\in f(a)$. Thus, $\exists x'\comp x$ $\forall y'\in Q(x')$ $\forall y\compflip y'$ $y\not\in f(a)$, which shows $x\not\in \Diamond f(a)$. \end{proof}

Dropping the completeness of $L$, we can prove the following result (note the proof can be carried out in ZF without the Axiom of Choice, in the spirit of \cite{Bezhanishvili-Holliday2020}), which embeds $L$ into its canonical extension (see \cite{Gehrke2001}). This is closely related to the topological representations of bounded lattices in \cite{Ploscica1995} and \cite{Craig2013}, building on \cite{Urquhart1978} and \cite{Allwein1993}, and the topological representation of Boolean algebras in \cite{Bezhanishvili-Holliday2020}. The treatment of negation was added in~\cite{Holliday2022}.

 \begin{theorem}\label{GenModalRep1} Let $L$ be a bounded lattice with $\neg$, $\Box$, and $\Diamond$ where
\begin{itemize}
\item $\neg$ is an antitone unary operation on $L$ with $\neg 1=0$,
\item $\Box$ is a multiplicative unary operation on $L$,  and 
\item $\Diamond$ is an additive unary operation on $L$. 
\end{itemize}
Then define: 
\begin{itemize}
\item $X=\{(F,I)\mid F \mbox{ is a filter in }L, I\mbox{ is an ideal in }L,\mbox{and } \{\neg a\mid a\in F\}\subseteq I\}$;
\item $(F,I)\comp (F',I')$ iff  $I\cap F'=\varnothing$;
\item $(F,I)R(F',I')$ iff for all $a\in L$, $\Box a\in F\Rightarrow a\in F'$;
\item $(F,I)Q(F',I')$ iff for all $a\in L$, $\Diamond a\in I\Rightarrow a\in I'$.
\end{itemize}
Then  $(X,\comp, R, Q)$ is an additive modal frame; $(L,\neg, \Box,\Diamond)$ embeds into ${(\lat(X,\comp),\neg_\comp, \Box_R,\Diamond_Q)}$; and $(L,\neg, \Box,\Diamond)$  is isomorphic to the subalgebra of ${(\lat(X,\comp),\neg_\comp, \Box_R,\Diamond_Q)}$ consisting of $c_\comp$-fixpoints that are compact open in the topology on $X$ generated by $\{\widehat{a}\mid a\in L\}$, where $\widehat{a}=\{(F,I)\mid a\in F\}$.
 \end{theorem}
 
 \begin{proof} For the following, given an element $a$ of a lattice, let $\mathord{\uparrow}a$ (resp.~$\mathord{\downarrow}a$) be the principal filter (resp.~ideal) generated by $a$.

We first show that $(X,\comp, R, Q)$ is a modal frame as in Definition \ref{ModalFrameDef}:
\[\mbox{if $x Ry\compflip z$, then $\exists x'\comp x$ $\forall x''\compflip x' $ $\exists y''$: $x''Ry''\compflip z$.}\]
The proof slightly adapts that of Proposition 4.10 of \cite{Holliday2022} to account for the role of $\neg$ in the definition of $X$. Suppose $(F,I)R(G,H)\compflip (J,K)$, which implies $K\cap G=\varnothing$ and hence $K\cap \{a \mid \Box a\in F\}=\varnothing$. Then where $F'=\mathord{\uparrow}1$ and $I'$ is the ideal generated by $\{\Box a\mid a\in K\}$, we claim that $ I'\cap F=\varnothing$, so $(F',I')\comp (F,I)$.  For if $b\in I'\cap F$, then for some  $a_1,\dots,a_n\in K$, we have $b\leq \Box a_1\vee\dots\vee \Box a_n$, which implies $b\leq \Box(a_1\vee\dots\vee a_n)$, so ${\Box(a_1\vee\dots\vee a_n)\in F}$, whence $a_1\vee\dots\vee a_n\not\in K$, contradicting $a_1,\dots,a_n\in K$. Now suppose ${(F',I')\comp (F'',I'')}$, so $I'\cap F''=\varnothing$. Let $G''=\{b\mid \Box b\in F''\}$, which is a filter,  and let $H''$ be the ideal generated by $\{\neg a\mid a\in G''\}$. We claim $K\cap G''=\varnothing$, so $(J,K)\comp (G'',H'')$. For if $a\in G''$, then $\Box a\in F''$, so $\Box a\not\in I'$, whence $a\not\in K$. Thus, $(F'',I'')R(G'',H'')\compflip (J,K)$, which establishes the condition.

Now we show that  $(X,\comp, R,Q)$ is additive as in Definition \ref{AddDef}:
\[\mbox{if $x Qy\comp z$, then $\exists x'\compflip x$ $\forall x''\comp x'$ $\exists y''$: $x''Qy''\comp z$}.\] 
Suppose $(F,I)Q(G,H)\comp (J,K)$, which implies $H\cap J=\varnothing$ and hence ${\{a \mid \Diamond a\in I\}\cap J=\varnothing}$. Let $F'$ be the filter generated by $\{\Diamond a\mid a\in J\}$ and $I'$ the ideal generated by $\{\neg a\mid a\in F'\}$, so $(F',I')\in X$. We claim that $I\cap F'=\varnothing$, so $(F',I')\compflip (F,I)$. If $b\in  F'$, then there are $a_1,\dots,a_n\in J$ such that $\Diamond a_1\wedge\dots\wedge\Diamond a_n\leq b$, which implies $\Diamond (a_1\wedge\dots\wedge a_n)\leq b$.  If in addition $b\in I$, then $\Diamond (a_1\wedge\dots\wedge a_n)\in I$, so $a_1\wedge\dots\wedge a_n\not\in J$, contradicting $a_1,\dots,a_n\in J$. Now consider any $(F'',I'')\comp (F',I')$, so $I''\cap F'=\varnothing$. Let $G''=\mathord{\uparrow}1$ and $H''=\{a \mid \Diamond a\in I''\}$, which is an ideal, so $(F'',I'')Q(G'',H'')$. We claim that $H''\cap J=\varnothing$, so $(G'',H'')\comp (J,K)$. For if $b\in H''$, then $\Diamond b\in I''$, which implies $\Diamond b\not\in F'$ and hence $b\not\in J$. This completes the proof of the condition.

The claimed embedding sends $a$ to $\widehat{a}$. We verify that it preserves $\Box$ and $\Diamond$. For its other claimed properties, see the proof of  Theorem 4.30 in \cite{Holliday2023}.\footnote{That proof assumes $F\cap I=\varnothing$ for each filter-ideal pair, in which case $\comp$ is reflexive, but the proof easily adapts to drop that assumption. Cf.~the proof of Theorem B.7 of \cite{Holliday2023}.}

Let us show $\widehat{\Box a}=\Box_R\widehat{a}$. Suppose $(F,I)\in \widehat{\Box a}$, so $\Box a\in F$. Then if $(F,I)R(F',I')$, we have $a\in F'$ and hence $(F',I')\in\widehat{a}$. Thus, $(F,I)\in\Box_R\widehat{a}$. Conversely, suppose $(F,I)\not\in \widehat{\Box a}$, so $\Box a\not\in F$. Let $F'=\{b \mid \Box b\in F\}$, which is a filter, and let $I'$ be the ideal generated by $\{\neg b\mid b\in F'\}$, so $(F',I')\in X$. Then $\Box a\not\in F$ implies $a\not\in F'$ and hence ${(F',I')\not\in\widehat{a}}$, and by construction of $F'$, we have $(F,I)R(F',I')$. Thus, $(F,I)\not\in\Box_R\widehat{a}$.

Finally, we show $\widehat{\Diamond a}=\Diamond_Q\widehat{a}$. Suppose $(F,I)\in\widehat{\Diamond a}$, so $\Diamond a\in F$.  Consider any $(F',I')\comp (F,I)$, so $\Diamond a\not\in I'$. Let $G'=\mathord{\uparrow}1$ and ${H'=\{b\in L\mid \Diamond b\in I'\}}$, which is an ideal, so $(F',I')Q(G',H')$. Let $G=\mathord{\uparrow}a$ and $H=\mathord{\downarrow}\neg a$, so $(G,H)\in\widehat{a}$. We claim that $H'\cap G=\varnothing$, so $(G',H')\comp (G,H)$. For otherwise $a\in H'$,  so $\Diamond a\in I'$, contradicting what we derived above. This proves that  $(F,I)\in\Diamond_Q\widehat{a}$. Conversely, suppose $(F,I)\not\in\widehat{\Diamond a}$, so $\Diamond a\not\in F$. Let $F'=\mathord{\uparrow}1$ and $I'=\mathord{\downarrow}\Diamond a$, so $(F',I')\comp (F,I)$. Consider any $(G',H')$ and $(G,H)$ such that $(F',I')Q(G',H')\comp (G,H)$. Then since $\Diamond a\in I'$, we have $a\in H'$ and hence $a\not\in G$, so $(G,H)\not\in\widehat{a}$. This proves that $(F,I)\not\in\Diamond_Q\widehat{a}$.\end{proof}

 \section{Interactions}\label{Interactions}
 
We now consider the interaction of $\Box$ and $\Diamond$ via $\neg$. Of course, if $\Box$ and $\Diamond$ come from different flavors of modality, e.g., $\Box a$ means that \textit{the agent believes $a$} and $\Diamond a$ means that $a$ \textit{will hold sometime in the future}, there need be no interaction between them via $\neg$. But even if $\Box$ and $\Diamond$ are of the same flavor of modality, the interactions between them via $\neg$ may be subtle (see Remark \ref{Unawareness}).
  
  One may also consider interactions between $\Box$ and $\Diamond$ via other operations, such as $\wedge$. In classical modal logic, we have $\Box a\wedge\Diamond b\leq \Diamond (a\wedge b)$. However, this is not desirable in epistemic orthologic \cite[Example~3.39]{Holliday-Mandelkern2022}, so we do not wish to impose this constraint. Of course, the simplest interaction to consider is $\Box a\leq\Diamond a$, but this cannot be imposed for doxastic logic with possibly inconsistent agents. By contrast, some interactions between $\Box$ and $\Diamond$ via $\neg$ seem generally acceptable---when $\Box$ and $\Diamond$ come from the same flavor of modality---and will allow us to simplify our semantics by setting $R=Q$ in \S~\ref{UnificationSection} (cf.~\cite{Prenosil2023} on when a single relation suffices for distributive modal logics). Thus, here we focus only on interactions via $\neg$ and leave the study of further interactions for future~work.
 
 \subsection{Lattice inequalities}\label{LIs}

Consider the following axioms, implicitly universally quantified:
 \begin{align}
 \Diamond\neg a\leq \neg\Box a\tag{$\Diamond\neg$}\label{DiamondNeg}\\
  \Box\neg a\leq \neg\Diamond a \tag{$\Box\neg$}\label{BoxNeg}\\
 \neg \Diamond a\leq \Box \neg a\tag{$\neg\Diamond$}\label{NegDiamond}\\
 \neg\Box a\leq \Diamond \neg a.\tag{$\neg\Box$}\label{NegBox}
 \end{align}
 \noindent First we observe that over the most general algebras considered in \S~\ref{AlgFrames}, the above axioms are all independent.
 
 \begin{proposition}\label{AllInd} Each of \textnormal{(}\ref{DiamondNeg}\textnormal{)}, \textnormal{(}\ref{BoxNeg}\textnormal{)}, \textnormal{(}\ref{NegDiamond}\textnormal{)}, and  \textnormal{(}\ref{NegBox}\textnormal{)} is independent of all the others over finite lattices equipped with an antitone $\neg$ sending $1$ to $0$,  multiplicative $\Box$, and additive $\Diamond$.
 \end{proposition}
 \begin{proof} The independence of (\ref{NegDiamond}) and (\ref{NegBox}) will be shown in Propositions \ref{NegDiamondInd}.(\ref{NegDiamondInd2}) and \ref{NegBoxInd}, respectively. For (\ref{DiamondNeg}), consider the following four-element lattice equipped with the following $\neg$, $\Box$, and $\Diamond$: 
    \begin{center}
    \begin{minipage}{1.5in}
     \begin{tikzpicture}[->,>=stealth',shorten >=1pt,shorten <=1pt, auto,node
distance=2cm,thick,every loop/.style={<-,shorten <=1pt}]
\tikzstyle{every state}=[fill=gray!20,draw=none,text=black]
\node (0) at (0,0) {{$0$}};
\node (a) at (-.75,.75) {{$a$}};
\node (b) at (.75,.75) {{$b$}};
\node (1) at (0,1.5) {{$1$}};

\path (1) edge[-] node {{}} (a);
\path (1) edge[-] node {{}} (b);
\path (a) edge[-] node {{}} (0);
\path (b) edge[-] node {{}} (0);
\end{tikzpicture}\end{minipage}\begin{minipage}{1.5in}\begin{tabular}{cccc}
 $x$ & $\neg x$ & $\Box x$ & $\Diamond x$\\
 \hline
 $1$ & $0$ & $1$ & $1$ \\
 $a$ & $0$ & $a$ & $1$ \\
  $b$ & $a$ & $b$ & $b$ \\
 $0$ & $1$ & $0$ & $0$
 \end{tabular}
 \end{minipage}
 \end{center}
 Then $\neg$ is antitone and sends $1$ to $0$, $\Box$ is multiplicative, $\Diamond$ is additive, and (\ref{BoxNeg}), (\ref{NegDiamond}), and (\ref{NegBox}) hold. However, $\Diamond \neg b = \Diamond a = 1\not\leq a = \neg b = \neg \Box b$, so  (\ref{DiamondNeg}) does not hold. For (\ref{BoxNeg}), consider the same lattice as above with the same $\neg$ but with the following $\Box$ and $\Diamond$:
 \begin{center}
 \begin{tabular}{cccc}
 $x$ & $\neg x$ & $\Box x$ & $\Diamond x$\\
 \hline
 $1$ & $0$ & $1$ & $1$ \\
 $a$ & $0$ & $1$ & $1$ \\
  $b$ & $a$ & $0$ & $1$ \\
 $0$ & $1$ & $0$ & $0$
 \end{tabular}
 \end{center}
  Then $\neg$ is antitone and sends $1$ to $0$, $\Box$ is multiplicative, $\Diamond$ is additive, and (\ref{DiamondNeg}), (\ref{NegDiamond}), and (\ref{NegBox}) hold. However, $\Box \neg b = \Box a =1 \not\leq 0 = \neg 1 = \neg \Diamond b$, so  (\ref{BoxNeg}) does not hold.\end{proof}
 
 \subsubsection{Interactions in fundamental logic}

In the context of fundamental logic, we collapse one distinction from~\S~\ref{LIs}.

 \begin{proposition}\label{DSAM} If $\neg$ is dually self-adjoint and $\Box$ and $\Diamond$ are monotone, then \textnormal{(}\ref{DiamondNeg}\textnormal{)} is equivalent to \textnormal{(}\ref{BoxNeg}\textnormal{)}.
 \end{proposition}
 \begin{proof} Assume $\neg$ is dually self-adjoint and hence antitone and double inflationary by Lemma \ref{DSA}.  Assuming $f$ is a monotone unary operation and $g$ an arbitrary unary operation, we prove that if for all $a\in L$, $f(\neg a)\leq \neg g(a)$, then for all $a\in L$, $g(\neg a)\leq \neg f(a)$, from which the statement in the lemma follows.
 
  Assume $f(\neg a)\leq \neg g(a)$ for all $a\in L$. By dual self-adjointness, $f(\neg a)\leq\neg g(a)$ implies $g(a)\leq\neg f( \neg a)$, so for all $a\in L$, we have $g(\neg a)\leq \neg f(\neg\neg a)$. Then since $a\leq\neg\neg a$, we have $f(a)\leq f(\neg\neg a)$ by the monotonicity of $f$, so $\neg f(\neg\neg a)\leq \neg f(a)$ by the antitonicity of $\neg$. Hence $g(\neg a)\leq \neg f(a)$.\end{proof}
  
 \begin{remark}\label{Unawareness} Consider the interpretation of the modalities where $\Box a$ means \textit{the agent is certain that $a$} and $\Diamond a$ means \textit{the agent considers it possible that~$a$}. Then (\ref{DiamondNeg}) and (\ref{BoxNeg}) are plausible. Yet  (\ref{NegDiamond}) is questionable: from the assumption that an agent does not consider it possible that $a$, it does not follow that the agent is certain that $\neg a$; for the agent may be totally \textit{unaware} of $a$, neither entertaining the possibility of $a$ nor having any attitude with the content $\neg a$. Similarly, (\ref{NegBox}) is questionable: from the fact that the agent is not certain that $a$, it does not follow that the agent considers it possible that $\neg a$, again because the agent may have unawareness. Thus, neither (\ref{NegDiamond}) nor  (\ref{NegBox}) belongs in a base system of fundamental modal logic.\end{remark}
  
  \subsubsection{Interactions in intuitionistic logic}
  
In intuitionistic modal logic, where $\neg$ is \textit{pseudocomplementation} ($a\wedge b=0$ implies $b\leq \neg a$, and $a\wedge\neg a=0$) and hence dually self-adjoint, it is standard to have not only (\ref{DiamondNeg}) and (\ref{BoxNeg}) but also (\ref{NegDiamond}), despite the concern about (\ref{NegDiamond}) in Remark~\ref{Unawareness}. Note that (\ref{NegDiamond}) is an additional condition, even classically.

\begin{proposition}\label{NegDiamondInd} $\,$
\begin{enumerate}
\item\label{NegDiamondInd1} \textnormal{(}\ref{NegDiamond}\textnormal{)} is independent of \textnormal{(}\ref{DiamondNeg}\textnormal{)} and \textnormal{(}\ref{BoxNeg}\textnormal{)} over Boolean algebras  equipped with a multiplicative $\Box$ and additive $\Diamond$.
\item\label{NegDiamondInd2} \textnormal{(}\ref{NegDiamond}\textnormal{)} is independent of \textnormal{(}\ref{DiamondNeg}\textnormal{)}, \textnormal{(}\ref{BoxNeg}\textnormal{)}, and \textnormal{(}\ref{NegBox}\textnormal{)} over Heyting algebras equipped with a multiplicative $\Box$ and additive $\Diamond$.
\end{enumerate} \end{proposition}
   \begin{proof}  For part (\ref{NegDiamondInd1}), consider the four-element Boolean algebra equipped with the following $\Box$ and $\Diamond$: 
    \begin{center}
    \begin{minipage}{1.5in}
     \begin{tikzpicture}[->,>=stealth',shorten >=1pt,shorten <=1pt, auto,node
distance=2cm,thick,every loop/.style={<-,shorten <=1pt}]
\tikzstyle{every state}=[fill=gray!20,draw=none,text=black]
\node (0) at (0,0) {{$0$}};
\node (a) at (-.75,.75) {{$a$}};
\node (b) at (.75,.75) {{$b$}};
\node (1) at (0,1.5) {{$1$}};

\path (1) edge[-] node {{}} (a);
\path (1) edge[-] node {{}} (b);
\path (a) edge[-] node {{}} (0);
\path (b) edge[-] node {{}} (0);
\end{tikzpicture}\end{minipage}\begin{minipage}{1.5in}\begin{tabular}{ccc}
 $x$ & $\Box x$ & $\Diamond x$\\
 \hline
 $1$ & $1$ & $1$ \\
 $a$ & $0$ & $1$ \\
  $b$ & $0$ & $0$ \\
 $0$ & $0$ & $0$
 \end{tabular}
 \end{minipage}
 \end{center}
   Then $\Box$ is multiplicative, $\Diamond$ is additive, $\Box x\leq\Diamond x$, and (\ref{DiamondNeg}) and (\ref{BoxNeg}) hold. However, we have $\neg \Diamond b =\neg 0=1\not\leq 0=\Box a=\Box\neg b$, so (\ref{NegDiamond}) does not hold. 
   
   For part (\ref{NegDiamondInd2}), consider the following five-element Heyting algebra with $\Box$ and $\Diamond$ operations:
   
   \begin{center}
   \begin{minipage}{1.5 in}
   \begin{tikzpicture}[->,>=stealth',shorten >=1pt,shorten <=1pt, auto,node
distance=2cm,thick,every loop/.style={<-,shorten <=1pt}]
\tikzstyle{every state}=[fill=gray!20,draw=none,text=black]
\node (0) at (0,0) {{$0$}};
\node (a) at (-.75,.75) {{$a$}};
\node (b) at (.75,.75) {{$b$}};
\node (1) at (0,1.5) {{$c$}};
\node  (new1) at (0,2.25) {{$1$}};

\path (new1) edge[-] node {{}} (1);
\path (1) edge[-] node {{}} (a);
\path (1) edge[-] node {{}} (b);
\path (a) edge[-] node {{}} (0);
\path (b) edge[-] node {{}} (0);
\end{tikzpicture} \end{minipage}\begin{minipage}{1.5in}\begin{tabular}{ccc}
 $x$ & $\Box x$ & $\Diamond x$\\
 \hline
 $1$ & $1$ & $1$ \\
  $c$ & $1$ & $1$ \\
 $a$ & $c$ & $1$ \\
  $b$ & $0$ & $0$ \\
 $0$ & $0$ & $0$
 \end{tabular}
 \end{minipage}
\end{center}
  Then $\Box$ is multiplicative, $\Diamond$ is additive, $\Box x\leq \Diamond x$, and not only (\ref{DiamondNeg}) and (\ref{BoxNeg}) but also (\ref{NegBox}) holds:
  \begin{itemize}
  \item $\neg \Box 1 = \neg 1= 0\leq \Diamond \neg 1$; $\neg \Box c =\neg 1=0\leq\Diamond \neg c $; $\neg \Box a = \neg c = 0 \leq \Diamond \neg a$;
  \item $\neg\Box b = \neg 0 = 1= \Diamond a = \Diamond \neg b$; $\neg\Box 0 = \neg 0 = 1 = \Diamond 1 = \Diamond \neg 0$.
  \end{itemize}
  However, we have $\neg \Diamond b = \neg 0 = 1\not\leq c = \Box a = \Box \neg b$, so (\ref{NegDiamond}) does not hold.   \end{proof}

 As noted in \S~\ref{Intro}, (\ref{NegBox}) is generally not  assumed in intuitionistic modal logic.
 
 \begin{proposition}\label{NegBoxInd} \textnormal{(}\ref{NegBox}\textnormal{)} is independent of \textnormal{(}\ref{DiamondNeg}\textnormal{)}, \textnormal{(}\ref{BoxNeg}\textnormal{)}, and \textnormal{(}\ref{NegDiamond}\textnormal{)} over Heyting algebras $H$ equipped with a multiplicative $\Box$ and additive $\Diamond$, even assuming that $\Box a\leq \Diamond a$ for all $a\in H$.
 \end{proposition}
 \begin{proof} Consider the three-element Heyting algebra equipped with the following $\Box$ and $\Diamond$:
 \begin{center}
 \begin{minipage}{1.25in}  \begin{tikzpicture}[->,>=stealth',shorten >=1pt,shorten <=1pt, auto,node
distance=2cm,thick,every loop/.style={<-,shorten <=1pt}]
\tikzstyle{every state}=[fill=gray!20,draw=none,text=black]
\node (0) at (0,0) {{$0$}};
\node (a) at (0,.75) {{$a$}};
\node  (1) at (0,1.5) {{$1$}};

\path (1) edge[-] node {{}} (a);

\path (a) edge[-] node {{}} (0);
\end{tikzpicture} 
 \end{minipage}\begin{minipage}{1.25in}
 \begin{tabular}{ccc}
 $x$ & $\Box x$ & $\Diamond x$\\
 \hline
 $1$ & $1$ & $1$ \\
 $a$ & $0$ & $1$ \\
 $0$ & $0$ & $0$
 \end{tabular}\end{minipage}
 \end{center}
 Then $\Box$ is multiplicative, $\Diamond$ is additive, and (\ref{DiamondNeg}), (\ref{BoxNeg}), and (\ref{NegDiamond})  hold. But $\neg \Box a = \neg 0 = 1\not\leq 0 = \Diamond 0 = \Diamond \neg a$, so (\ref{NegBox}) does not hold.\end{proof}
 
 \subsubsection{Interactions in orthologic}
  
 Finally, in the context of modal orthologic, where $\neg$ is involutive, it is natural to take $\Diamond$ and $\Box$ to be duals in the following sense:
  \begin{align}
  \Diamond a   &= \neg\Box\neg a \tag{$\Diamond$ def}\label{DiamondDef}\\
  \Box a  &= \neg\Diamond\neg a \tag{$\Box$ def} \label{BoxDef}.
  \end{align}
  
  \begin{proposition}\label{InvLemm} If $\neg$ is antitone and involutive and $\Box$, $\Diamond$ are monotone, then:
  \begin{enumerate}
  \item\label{InvLemmA} \textnormal{(}\ref{DiamondNeg}\textnormal{)} and  \textnormal{(}\ref{BoxNeg}\textnormal{)} are equivalent;
  \item\label{InvLemmB} \textnormal{(}\ref{NegDiamond}\textnormal{)} and \textnormal{(}\ref{NegBox}\textnormal{)} are equivalent; 
  \item\label{InvLemmC} \textnormal{(}\ref{DiamondDef}\textnormal{)} and \textnormal{(}\ref{BoxDef}\textnormal{)} are equivalent to each other and to the conjunction of \textnormal{(}\ref{DiamondNeg}\textnormal{)}, \textnormal{(}\ref{BoxNeg}\textnormal{)}, \textnormal{(}\ref{NegDiamond}\textnormal{)}, and \textnormal{(}\ref{NegBox}\textnormal{)}.
  \end{enumerate}
  \end{proposition}
  
  \begin{proof} For part (\ref{InvLemmA}), since $\neg$ is antitone and involutive, it is dually self-adjoint by Lemma \ref{DSA}, so (\ref{DiamondNeg}) and (\ref{BoxNeg}) are equivalent by Lemma~\ref{DSAM}. For part (\ref{InvLemmB}), assume (\ref{NegDiamond}).  As an instance, we have $ \neg \Diamond \neg a\leq \Box \neg \neg a$, which implies $\neg\Box a\leq \Diamond \neg a$ by antitonicity and involution, so  (\ref{NegBox}) holds. Now assume (\ref{NegBox}). As an instance, we have $\neg\Box\neg a\leq \Diamond\neg\neg a$, which implies $\neg\Diamond a\leq \Box\neg a$ by antitonicity and involution, so (\ref{NegDiamond}) holds. For part (\ref{InvLemmC}), clearly  (\ref{DiamondDef}) and (\ref{BoxDef}) are equivalent given involution and imply (\ref{DiamondNeg}), (\ref{BoxNeg}), (\ref{NegDiamond}), and (\ref{NegBox}) given involution. Conversely, assume (\ref{DiamondNeg}), (\ref{BoxNeg}), (\ref{NegDiamond}), and (\ref{NegBox}).   By (\ref{NegDiamond}), we have $\neg\Diamond a\leq \Box \neg a$, which implies $\neg\Box\neg a\leq\Diamond a$ by antitonicity and involution. By (\ref{BoxNeg}), we have $\Box\neg a\leq \neg\Diamond a$, which implies $\Diamond a\leq \neg\Box\neg a$ by antitonicity and involution. Hence $\Diamond a=\neg\Box\neg a$, so (\ref{DiamondDef}) holds.   \end{proof}
  
 \subsection{Frame conditions}
 
 Let us now identify frame conditions sufficient for the principle (\ref{DiamondNeg}), which seems unobjectionable (when the same flavor of modality is involved on both sides), and (\ref{NegDiamond}), which is typically assumed in intuitionistic modal logic.

  \begin{proposition}\label{DiamondNegProp} Let $(X,\comp,R,Q)$  be a modal frame in which $Q\subseteq R$.  Then for every $c_\comp$-fixpoint $A$, \[\Diamond\neg A\subseteq \neg\Box A.\]
\end{proposition}
\begin{proof} Suppose $x\in\Diamond\neg A$. Toward a contradiction, suppose $x\not\in \neg\Box A$, so there is a $y\comp x$ such that $y\in \Box A$. Since $y\comp x$ and $x\in\Diamond\neg A$, there is a $z\in Q(y)$  and $w\compflip z$ with $w\in \neg A$, which implies $z\not\in A$, which contradicts the facts that $y\in \Box A$ and~$yQz$, since by our assumption $yQz$ implies $yRz$.\end{proof}

 \begin{definition}\label{NegativeFrameDef} A modal frame $(X,\comp,R,Q)$ is \textit{negative} if for all $x,y,z\in X$,
\[\mbox{if $xRy\compflip z$, then $\exists x'\comp x$ $\forall x''\comp x'$ $\exists y''$: $x''Qy''\comp z$.}\]
 \end{definition}
 
 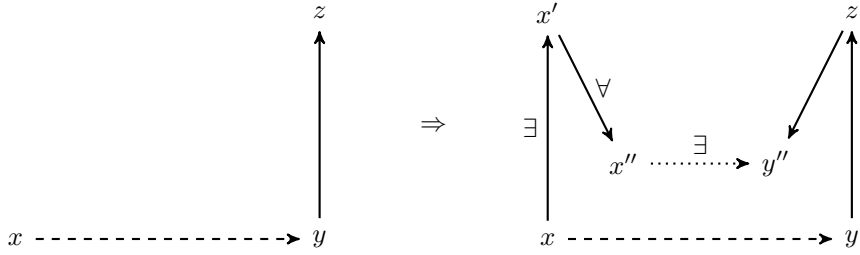
\begin{figure}[h]
\begin{center}
\begin{tikzpicture}[->,>=stealth',shorten >=1pt,shorten <=1pt, auto,node
distance=2cm,thick,every loop/.style={<-,shorten <=1pt}]
\tikzstyle{every state}=[fill=gray!20,draw=none,text=black]

\node (x) at (0,0) {{$x$}};
\node (y) at (4,0) {{$y$}};
\node (z) at (4,3) {{$z$}};
\node at (5.5,1.5) {{\textit{$\Rightarrow$}}};

\path (x) edge[dashed,->] node {{}} (y);
\path (y) edge[->] node {{}} (z);

\node (x2) at (7,0) {{$x$}};
\node (x'2) at (7,3) {{$x'$}};
\node (x''2) at (8,1) {{$x''$}};
\node (y''2) at (10,1) {{$y''$}};
\node (y2) at (11,0) {{$y$}};
\node (z2) at (11,3) {{$z$}};

\path (x2) edge[dashed,->] node {{}} (y2);
\path (x''2) edge[dotted,->] node[above] {{$\exists$}} (y''2);
\path (y2) edge[->] node {{}} (z2);
\path (x2) edge[->] node {{$\exists$}} (x'2);
\path (x''2) edge[<-] node[right] {{$\forall$}} (x'2);
\path (y''2) edge[<-] node {{}} (z2);

\end{tikzpicture}
\end{center}
\caption{Illustration of the negativity condition in Definition \ref{NegativeFrameDef}.}
\end{figure}

\begin{proposition}\label{NegDiamondBoxNeg} If $(X,\comp,R,Q)$  is a negative modal frame, then for every $c_\comp$-fixpoint $A$,
\[\neg\Diamond A\subseteq\Box\neg A.\]
\end{proposition}
\begin{proof} Suppose $x\not\in \Box\neg A$, so there is a $y\in X$ with $xRy$ and $z\comp y$ with $z\in A$. Then where $x'$ is as in Definition \ref{NegativeFrameDef}, we have $x'\in \Diamond A$, so $x\not\in \neg\Diamond A$.\end{proof}

\section{Unification}\label{UnificationSection}

By assuming (\ref{DiamondNeg}) and the dual self-adjointness of $\neg$, we can simplify the representation from Theorem \ref{CompModalRep}.

\begin{definition} A modal frame  $(X,\comp,R,Q)$ is \textit{unified} if $R=Q$.
\end{definition}

\begin{theorem}\label{CompModalRep2} Let $L$ be a complete lattice with $\neg$, $\Box$, and $\Diamond$ where
\begin{itemize}
\item $\neg$ is a dually self-adjoint unary operation on $L$ with $\neg 1=0$,
\item $\Box$ is a completely multiplicative unary operation on $L$, 
\item $\Diamond$ is a completely additive unary operation on $L$, and 
\item $\Diamond \neg a\leq\neg\Box a$ for all $a\in L$.
\end{itemize} 
Then define:
\begin{itemize}
\item $X= \{(a,b)\mid a,b\in L,\neg a\leq b\}$;
\item $x\comp y$ iff $x_1\not\geq y_0$;
\item $xRy$ iff  for all $a\in L$, $x_0\leq \Box a\Rightarrow y_0\leq a$ and $\Diamond a\leq x_1\Rightarrow  a\leq y_1$;
\item $Q  =  R$.
\end{itemize}
Then:
\begin{enumerate}
 \item\label{CompModalRep2a} $\mathcal{F}=(X,\comp, R,Q)$ is a unified, additive modal frame with $\comp$ pseudo-symmetric;
 \item\label{CompModalRep2b} $(L,\neg, \Box,\Diamond)$ is isomorphic to ${(\lat(X,\comp),\neg_\comp, \Box_R,\Diamond_Q)}$;
  \item\label{CompModalRep2d} if $a\wedge\neg a=0$ for all $a\in L$, then $\comp$ is pseudo-reflexive;
 \item\label{CompModalRep2c} if $\neg\Diamond a\leq \Box\neg a$ for all $a\in L$, then $\mathcal{F}$ is negative.
 \end{enumerate}
\end{theorem}

\begin{proof}  Given any $x\in X$, let 
  \[\tau(x)= (\bigwedge\{b\mid x_0\leq\Box b\},   \bigvee\{b\mid \Diamond b\leq x_1\}\vee  \neg \bigwedge\{b\mid x_0\leq\Box b\}).\]
Then observe the following:
\begin{itemize}
\item[(a)] We have $xR\tau(x)$. For $x_0\leq \Box c$ implies $\tau(x)_0\leq c$; and if $c\not\leq  \tau(x)_1$, then $ c\not\leq \bigvee\{b\mid \Diamond b\leq x_1\}$ and hence $\Diamond c\not\leq x_1$. 
\item[(b)] If $x_0\not\leq \Box a$, then $\tau(x)_0\not\leq a$, by the same reasoning as in the proof of Theorem~\ref{CompModalRep}.
\item[(c)] If $\Diamond a\not\leq x_1$, then  $a\not\leq \tau(x)_1$. Contrapositively,
\begin{eqnarray*}
&&a\leq \bigvee\{b\mid \Diamond b\leq x_1\}\vee  \neg \bigwedge\{b\mid x_0\leq\Box b\} \\
&\Rightarrow&  \Diamond a\leq \Diamond \Big(\bigvee\{b\mid \Diamond b\leq x_1\}\vee  \neg \bigwedge\{b\mid x_0\leq\Box b\}\Big)\\&&\mbox{by monotonicity of $\Diamond$}\\
&\Rightarrow&  \Diamond a\leq  \bigvee\{\Diamond b\mid \Diamond b\leq x_1\}\vee \Diamond \neg \bigwedge\{b\mid x_0\leq\Box b\}\\
&&\mbox{by complete additivity of $\Diamond$}\\
&\Rightarrow&  \Diamond a\leq  x_1\vee \Diamond \neg \bigwedge\{b\mid x_0\leq\Box b\}  \\
&\Rightarrow&  \Diamond a\leq  x_1\vee  \neg \Box \bigwedge\{b\mid x_0\leq\Box b\}\quad\mbox{since $\Diamond \neg d\leq \neg\Box d$} \\
&\Rightarrow&  \Diamond a\leq  x_1\vee  \neg  \bigwedge\{\Box b\mid x_0\leq\Box b\}\quad\mbox{by complete multiplicativity of $\Box$} \\
&\Rightarrow&  \Diamond a\leq  x_1\vee  \neg x_0 \quad\mbox{by antitonicity of $\neg$}\\
&\Rightarrow&  \Diamond a\leq  x_1\quad\mbox{since $\neg v_0\leq v_1$ for all $v\in X$.}
\end{eqnarray*}
\end{itemize}

  Now for part (\ref{CompModalRep2a}), that $\mathcal{F}$ is unified is immediate from the definition. The proof that $\mathcal{F}$ is an additive modal frame is almost exactly as in the proof of Theorem \ref{CompModalRep}, only using $\tau$ instead of $\rho$ and $\sigma$:
  
  First we show that $\mathcal{F}$ is a modal frame:
if $x Ry\compflip z$, then $\exists x'\comp x$ $\forall x''\compflip x' $ $\exists y''$: $x''Ry''\compflip z$. Suppose $x Ry\compflip z$. Then $x_0\not\leq \Box z_1$, for otherwise $xRy$ implies $y_0\leq z_1$, contradicting $y\compflip z$.  Now let $x'=(1, \Box z_1)$, so $x'\comp x$. Consider any $x''\compflip x'$, so $x''_0\not\leq x'_1=\Box z_1$. Let $y''=\tau(x'')$, so $x''Ry''$ by (a) above. Then $x''_0\not\leq \Box z_1$ implies $y''_0\not\leq z_1$ by (b) above, so $y''\compflip z$.

Next we show that $\mathcal{F}$ is additive: if $x Qy\comp z$, then $\exists x'\compflip x$ $\forall x''\comp x'$ $\exists y''$: $x''Qy''\comp z$. Suppose $x Qy\comp z$. Since $y\comp z$, we have $z_0\not\leq y_1$, which with $xQy$ implies $\Diamond z_0\not\leq x_1$. Then where $x'=(\Diamond z_0,\neg \Diamond z_0)$, we have $x'\compflip x$. Now consider any $x''\comp x'$, so $\Diamond z_0\not\leq x_1''$. Let $y''=\tau(x'')$, so $x''Ry''$ by (a) above. Then since $\Diamond z_0\not\leq x_1''$, we have $z_0\not\leq y''_1$ by (c) above. Hence $y''\comp z$.
  
That $\comp$ is pseudo-symmetric follows from the isomorphism in part (\ref{CompModalRep2b}) and Proposition \ref{CorrespondenceProp}(\ref{CorrespondenceProp2}). 

The proof of part (\ref{CompModalRep2b}) is almost exactly as in the proof of Theorem \ref{CompModalRep}, only using $\tau$ in place of $\rho$ and $\sigma$:

Exactly as in the proof of Theorem \ref{CompModalRep}, the function $f$ defined by $f(a)=\{x\in P\mid x_0\leq a\}$ is an isomorphism from $L$ to $\lat(X,\comp)$ that also preserves $\neg$. It only remains to check $\Box$ and~$\Diamond$.

To show that $f(\Box a)=\Box_Rf(a)$, first suppose $x\in f(\Box a)$, so $x_0\leq \Box a$. Then $xRy$ implies $y_0\leq a$ and hence $y_0\in f(a)$. Thus, $x\in\Box_Rf(a)$. Conversely, suppose $x\not\in f(\Box a)$, so $x_0\not\leq \Box a$. Let $y=\tau(x)$, so $xRy$ by (a) above. Then $x_0\not\leq \Box a$ implies $y_0\not\leq a$ by (b) above,  so $y\not\in f(a)$ and hence $x\not\in \Box_R f(a)$. 

To show that $f(\Diamond a)=\Diamond_Qf(a)$, first suppose $x\in f(\Diamond a)$, so $x_0\leq \Diamond a$. Further suppose $x'\comp x$, so $x_0\not\leq x'_1$ and hence $\Diamond a\not\leq x'_1$. Let $y'=\tau(x')$, so $x'Qy'$ by (a) above. Then since $\Diamond a\not\leq x'_1$, we have $a\not\leq y'_1$ by (c) above, so $y'\comp (a,\neg a)$. Thus, for all $x'\comp x$, there are $y',z'$ such that $x'Qy'\comp z'\in f(a)$. Hence $x\in\Diamond_Q f(a)$. Conversely, suppose $x\not\in f(\Diamond a)$, so $x_0\not\leq \Diamond a$. Let $x'=(1,\Diamond a)$, so $x'\comp x$. Now consider any $y'$ such that $x'Qy'$, which with $\Diamond a\leq x'_1$ implies $a\leq y'_1$. Then for any $y\compflip y'$, we have $y_0\not\leq y'_1$ and hence $y_0\not\leq a$, so $y\not\in f(a)$. Thus, $\exists x'\comp x$ $\forall y'\in Q(x')$ $\forall y\compflip y'$: $y\not\in f(a)$, which shows $x\not\in \Diamond f(a)$.

For part (\ref{CompModalRep2d}), assuming $a\wedge\neg a=0$ for all $a\in L$, that $\comp$ is pseudo-reflexive follows from the isomorphism in part (\ref{CompModalRep2b}) and Proposition \ref{CorrespondenceProp}(\ref{CorrespondenceProp1}). 

For part (\ref{CompModalRep2c}), assuming $\neg\Diamond a\leq\Box\neg a$ for all $a\in L$, we must show  negativity: if $xRy\compflip z$, then $\exists x'\comp x$ $\forall x''\comp x'$ $\exists y''$: $x''Qy''\comp z$. Suppose $xRy\compflip z$, which implies $x_0\not\leq \Box z_1$. Since $\neg z_0\leq z_1$, we have $\Box \neg z_0\leq \Box z_1$ by the monotonicity of $\Box$, so $x_0\not\leq \Box z_1$  implies $x_0\not\leq \Box \neg z_0$. Hence by our initial assumption, $x_0\not\leq \neg \Diamond  z_0$. Then where $x'=(\Diamond z_0,\neg\Diamond z_0)$, we have $x'\comp x$. Then since $x'\in f(\Diamond z_0)$ and $f$ preserves $\Diamond$, we have $x'\in \Diamond_Q f(z_0)$, which implies that $\forall x''\comp x'$ $\exists y''\in Q(x'')$ $\exists y'\compflip y''$: $y'\in f(z_0)$, so $y'_0\leq z_0$. From $y'\compflip y''$, we have $y'_0\not\leq y''_1$, which with $y'_0\leq z_0$ implies $z_0\not\leq y''_1$, so $y''\comp z$. Thus,  we have shown that $\forall x''\comp x'$ $\exists y''$: $x''Qy''\comp z$, as desired.\end{proof}

Similarly, we have the following unified analogue of Theorem \ref{GenModalRep1}.

\begin{theorem}\label{GenModalRep2} Let $L$ be a bounded lattice with $\neg$, $\Box$, and $\Diamond$ where
\begin{itemize}
\item $\neg$ is a dually self-adjoint unary operation on $L$ with $\neg 1=0$,
\item $\Box$ is a multiplicative unary operation on $L$,  
\item $\Diamond$ is an additive unary operation on $L$, and
\item $\Diamond \neg a\leq\neg\Box a$ for all $a\in L$.
\end{itemize}
Then define:
\begin{itemize}
\item $X=\{(F,I)\mid F \mbox{ is a filter in }L, I\mbox{ is an ideal in }L,\mbox{and } \{\neg a\mid a\in F\}\subseteq I\}$;
\item $(F,I)\comp (F',I')$ iff  $I\cap F'=\varnothing$;
\item $(F,I)R(F',I')$ iff for all $a\in L$, $\Box a\in F\Rightarrow a\in F'$ and $\Diamond a\in I\Rightarrow a\in I'$;
\item $Q=R$.
\end{itemize}
Then:
\begin{enumerate}
 \item\label{GenModalRep2a} $\mathcal{F}=(X,\comp, R, Q)$ is a unified, additive modal frame with $\comp$ pseudo-symmetric;
 \item\label{GenModalRep2b} there is an embedding of $(L,\neg, \Box,\Diamond)$ into ${(\lat(X,\comp),\neg_\comp, \Box_R,\Diamond_Q)}$ and an isomorphism between $(L,\neg, \Box,\Diamond)$  and the subalgebra of ${(\lat(X,\comp),\neg_\comp, \Box_R,\Diamond_Q)}$ consisting of $c_\comp$-fixpoints that are compact open in the topology on $X$ generated by $\{\widehat{a}\mid a\in L\}$, where $\widehat{a}=\{(F,I)\mid a\in F\}$;
 \item\label{GenModalRep2d}  if $a\wedge \neg a=0$ for all $a\in L$, then $\mathcal{F}$ is pseudo-reflexive;
  \item\label{GenModalRep2c}  if $\neg\Diamond a\leq \Box\neg a$ for all $a\in L$, then $\mathcal{F}$ is negative.
 \end{enumerate}
\end{theorem}

\begin{proof} For part (\ref{GenModalRep2a}), that $\mathcal{F}$ is unified is immediate from the definition. For the pseudo-symmetry of $\comp$, the proof is the same as in the proof of Proposition 4.32 in \cite{Holliday2023} (only without claiming  $F\cap I''=\varnothing$, which we do not need here). 

For the other properties, we explain how to modify the proof of Theorem~\ref{GenModalRep1} in light of the modified definitions of $R$ and $Q$ in Theorem \ref{GenModalRep2}. To show that $\mathcal{F}$ is a modal frame, modify the proof of the modal frame condition for Theorem~\ref{GenModalRep1} as follows: let $H''=\{a\mid \Diamond a\in I''\}$, which is an ideal. Then if $a\in G''$, we have $\Box b\in F''$, so $\neg \Box b\in I''$ and hence $\Diamond \neg b\in I''$ by the fourth bullet point of Theorem \ref{GenModalRep2}, so $\neg b\in H''$. Thus, $(G'',H'')\in X$.

To show $\mathcal{F}$ is additive, modify the proof of additivity for Theorem \ref{GenModalRep1} as follows: let $G''=\{a\mid \Box a\in F''\}$, which is a filter. Where $H''=\{a \mid \Diamond a\in I''\}$, which is an ideal, we have $(G'',H'')\in X$ by the same reasoning as in the previous paragraph. Then we have $(F'',I'')Q(G'',H'')$ by construction.

For the proof that $\widehat{\Box a}=\Box_R\widehat{a}$, modify the proof in Theorem \ref{GenModalRep1} by setting  $I'=\{b\mid \Diamond b\in I\}$. For the proof that $\widehat{\Diamond a}=\Diamond_Q\widehat{a}$, modify the proof in Theorem~\ref{GenModalRep1} by setting $G'=\{b\mid \Box b\in F'\}$.

For part (\ref{GenModalRep2d}), assuming $a\wedge\neg a=0$ for all $a\in L$, consider a non-absurd $(F,I)$, so there is some $(G,H)\comp (F,I)$. Hence $H\cap F=\varnothing$, so $0\not\in F$. Let $I'=\{a\mid \neg a\in F\}$. Then $F\cap I'=\varnothing$, for otherwise we have $a,\neg a\in F$ and hence $0\in F$, contradicting what we previously derived. Thus, $(F,I')\comp (F,I)$, and $(F,I')$ pre-refines $(F,I)$. This shows that $\comp$ is pseudo-reflexive. 

For part (\ref{GenModalRep2c}), assuming $\neg\Diamond a\leq\Box\neg a$, we show that $\mathcal{F}$ is negative:
\[\mbox{if $xRy\compflip z$, then $\exists x'\comp x$ $\forall x''\comp x'$ $\exists y''$: $x''Qy''\comp z$.}\]
Suppose $(F,I)R(G,H)\compflip (J,K)$. Let $F'$ be the filter generated by $\{\Diamond a\mid a\in J\}$ and $I'$ the ideal generated by $\{\neg\Diamond a\mid a\in J\}$, which is equal to $\{\Box \neg a\mid a\in J\}$ given $\neg\Diamond a\leq\Box\neg a$ and the converse from Proposition \ref{DSAM}. We claim that $I'\cap F=\varnothing$. For if $b\in I'$, then $b\leq \Box\neg a_1\vee\dots\vee\Box\neg a_n\leq \Box \neg (a_1\wedge\dots \wedge a_n)$ for some $a_1,\dots,a_n\in J$; and then if $b\in F$, we have $\Box \neg (a_1\wedge\dots \wedge a_n)\in F$, so $\neg (a_1\wedge\dots \wedge a_n)\in G$, which implies $ \neg (a_1\wedge\dots \wedge a_n)\not\in K$ and hence $a_1\wedge\dots\wedge a_n\not\in J$, contradicting $a_1,\dots,a_n\in J$. Thus, $I'\cap F=\varnothing$ and hence $(F',I')\comp (F,I)$. Now consider any $(F'',I'')\comp (F',I')$, so $I''\cap F'=\varnothing$. Let $G''=\{a\mid \Box a\in F''\}$ and $H''=\{a\mid \Diamond a\in I''\}$, so $(G'',H'')\in X$ as in the second paragraph of the proof above, and $(F'',I'')Q(G'',H'')$. We claim that $H''\cap J=\varnothing$. For if $a\in J$, then $\Diamond a\in F'$, which implies $\Diamond a\not\in I''$, which in turn implies $a\not\in H''$. Thus, $(G'',H'')\comp (J,K)$, which completes the proof.\end{proof}

\section{Fundamental modal logic}\label{FML}

Let $\mathcal{ML}$ be the propositional modal language with $\wedge,\vee,\neg,\Box,\Diamond$ and now also $\bot,\top$. At last, we define our proposed system of fundamental modal logic.
\begin{definition}\label{BinaryLogic} \textit{Fundamental modal logic} is the smallest binary relation \\ $\vdash \,\subseteq\mathcal{ML}\times\mathcal{ML}$ such that  for all $\varphi,\psi,\chi\in\mathcal{ML}$, not only conditions 1-11 of Definition \ref{LogicDef} but also the following hold:

\begin{center}
\begin{tabular}{ll}
12. $\bot\vdash\varphi\vdash\top$ & 17. $\top\vdash \Box\top$  \\
13. $\neg \top \vdash\bot$ &  18. $\Diamond \bot \vdash \bot$  \\ 

14. $\Box \varphi \wedge\Box \psi \vdash \Box (\varphi\wedge\psi)\quad $ & 19. if $\varphi\vdash\psi$, then $\Box\varphi\vdash\Box\psi$ \\
15. $\Diamond(\varphi\vee\psi) \vdash \Diamond\varphi\vee\Diamond\psi $ & 20. if $\varphi\vdash\psi$, then $\Diamond\varphi\vdash\Diamond\psi$.\\
16. $\Diamond \neg\varphi\vdash \neg\Box\varphi$
\end{tabular}
\end{center}
\end{definition}

\noindent The forcing clauses for $\Box\varphi$ and $\Diamond\varphi$ in relational models are as in Proposition~\ref{BoxDiamond}; $\bot$ is forced only at absurd states (recall \S~\ref{RelSem}), while $\top$ is forced at all states.

\begin{theorem}  Fundamental modal logic is sound and complete with respect to the class of unified, additive modal frames $(X,\comp,R)$ in which $\comp$ is pseudo-reflexive and pseudo-symmetric.
\end{theorem}
\begin{proof} Soundness is by Propositions \ref{RelSemFacts}, \ref{BoxDiamond}, \ref{AddLemm}, and \ref{DiamondNegProp}. For completeness, apply Theorem \ref{GenModalRep2} to the Lindenbaum-Tarski algebra of the logic.
\end{proof}

\section{Conclusion}\label{Conclusion}

We have proposed a way of adding modalities to fundamental logic, both axiomatically and semantically. Our representation theorems raise obvious questions about associated categorical dualities (see \cite{Massas2024} for morphisms), and the interactions between $\Box$ and $\Diamond$ via $\neg$ cry out for systematic correspondence theory. Also conspicuously absent has been ``the'' conditional $\to$. Weak conditionals possibly appropriate for fundamental logic are discussed in \cite[\S~6]{Holliday2023} and \cite{Holliday2024}. Treating a language with both modalities and conditionals is a natural next step, especially in connection with applications to natural language as in \cite{Holliday-Mandelkern2022}. Finally, our focus here has been entirely semantical. Yet we hope that in light of recent proof-theoretic successes with fundamental logic  \cite{Aguilera2022}, proof theorists might also find fundamental modal logic to be a worthy object of~study.

\subsection*{Acknowledgements}

I thank Yifeng Ding, Daniel Gonzalez, Guillaume Massas, and Yanjing Wang for helpful discussion and the three anonymous referees for helpful comments.

\bibliographystyle{aiml}
\bibliography{fundamentalmodal}

\providecommand{\noopsort}[1]{}\newcommand{\SortNoop}[1]{}
\begin{thebibliography}{10}
\expandafter\ifx\csname url\endcsname\relax
  \def\url#1{\texttt{#1}}\fi
\expandafter\ifx\csname urlprefix\endcsname\relax\def\urlprefix{URL }\fi
\newcommand{\enquote}[1]{``#1''}

\bibitem{Aguilera2022}
Aguilera, J.~P. and J.~Byd\u{z}ovsk\'y, \emph{Fundamental logic is decidable},
  ACM Transactions on Computational Logic  (Forthcoming),
  \href{https://doi.org/10.1145/3665328}{https://doi.org/10.1145/3665328}.

\bibitem{Allwein1993}
Allwein, G. and C.~Hartonas, \emph{Duality for bounded lattices} (1993),
  {I}ndiana University Logic Group, Preprint Series, IULG-93-25 (1993).

\bibitem{Almeida2009}
Almeida, A., \emph{Canonical extensions and relational representations of
  lattices with negation}, Studia Logica \textbf{91} (2009), pp.~171--199.

\bibitem{Battilotti1999}
Battilotti, G. and G.~Sambin, \emph{Basic logic and the cube of its
  extensions}, in: A.~Cantini, E.~Casari and P.~Minari, editors, \emph{Logic
  and Foundations of Mathematics},  Synthese Library  \textbf{280}, Kluwer
  Academic Publishers, 1999 pp. 165--186.

\bibitem{Bezhanishvili2024}
Bezhanishvili, N., A.~Dmitrieva, J.~de~Groot and T.~Moraschini, \emph{Positive
  modal logic beyond distributivity}, Annals of Pure and Applied Logic
  \textbf{175} (2024), p.~103374.

\bibitem{Bezhanishvili-Holliday2020}
Bezhanishvili, N. and W.~H. Holliday, \emph{Choice-free {S}tone duality}, The
  Journal of Symbolic Logic \textbf{85} (2020), pp.~109--148.

\bibitem{Birkhoff1940}
Birkhoff, G., \enquote{Lattice Theory,} American Mathematical Society, New
  York, 1940.

\bibitem{Bobzien2020}
Bobzien, S. and I.~Rumfitt, \emph{Intuitionism and the modal logic of
  vagueness}, Journal of Philosophical Logic \textbf{49} (2020), pp.~221--248.

\bibitem{BozicDosen1984}
Bo{\v{z}}i{\'c}, M. and K.~Do{\v{s}}en, \emph{Models for normal intuitionistic
  modal logics}, Studia Logica \textbf{43} (1984), pp.~217--245.

\bibitem{Conradie2019}
Conradie, W., A.~Craig, A.~Palmigiano and N.~M. Wijnberg, \emph{Modelling
  informational entropy}, in: R.~Iemhoff, M.~Moortgat and R.~Queiroz, editors,
  \emph{Logic, Language, Information, and Computation. WoLLIC 2019},  Lectures
  Notes in Computer Science  \textbf{11541}, 2019, pp. 140--160.

\bibitem{Conradie2016}
Conradie, W., S.~Frittella, A.~Palmigiano, M.~Piazzai, A.~Tzimoulis and N.~M.
  Wijnberg, \emph{Categories: {H}ow {I} learned to stop worrying and love two
  sorts}, in: J.~V\"{a}\"{a}n\"{a}nen, A.~Hirvonen and R.~de~Queiroz, editors,
  \emph{Logic, Language, Information, and Computation. WoLLIC 2016},  Lectures
  Notes in Computer Science  \textbf{9803}, 2016, pp. 145--164.

\bibitem{Craig2013}
Craig, A. P.~K., M.~Haviar and H.~A. Priestley, \emph{A fresh perspective on
  canonical extensions for bounded lattices}, Applied Categorical Structures
  \textbf{21} (2013), pp.~725--749.

\bibitem{Chiara2002}
{Dalla Chiara}, M.~L. and R.~Giuntini, \emph{Quantum logics}, in: D.~Gabbay and
  F.~Guenthner, editors, \emph{Handbook of Philosophical Logic}, Springer, 2002
  pp. 129--228.

\bibitem{Dmitrieva2021}
Dmitrieva, A., \enquote{Positive modal logic beyond distributivity: duality,
  preservation and completeness,} Master's thesis, University of Amsterdam
  (2021).

\bibitem{Dosen1984}
Do\v{s}en, K., \emph{Negative modal operators in intuitionistic logic},
  Publications de l'Institut Math\'{e}matique. Nouvelle S\'{e}rie \textbf{35}
  (1984), pp.~3--14.

\bibitem{Dosen1986}
Do\v{s}en, K., \emph{Negation as a modal operator}, Reports on Mathematical
  Logic \textbf{20} (1986), pp.~15--27.

\bibitem{Dosen1999}
Do\v{s}en, K., \emph{Negation in the light of modal logic}, in: D.~M. Gabbay
  and H.~Wansing, editors, \emph{What is Negation?}, Kluwer, Dordrecht, 1999
  pp. 77--86.

\bibitem{Dunn1993}
Dunn, J.~M., \emph{Star and perp: {T}wo treatments of negation}, Philosophical
  Perspectives \textbf{7} (1993), pp.~331--357.

\bibitem{Dunn1996}
Dunn, J.~M., \emph{Generalized ortho negation}, in: H.~Wansing, editor,
  \emph{Negation. A Notion in Focus}, de Gruyter, Berlin, 1996 pp. 3--26.

\bibitem{Dunn1999}
Dunn, J.~M., \emph{A comparative study of various model-theoretic treatments of
  negation: a history of formal negation}, in: D.~M. Gabbay and H.~Wansing,
  editors, \emph{What is Negation?}, Kluwer, Dordrecht, 1999 pp. 23--51.

\bibitem{Dunn2005}
Dunn, J.~M. and C.~Zhou, \emph{Negation in the context of gaggle theory},
  Studia Logica \textbf{80} (2005), pp.~235--264.

\bibitem{Dzik2006}
Dzik, W., E.~Orlowska and C.~van Alten, \emph{Relational representation
  theorems for general lattices with negations}, in: \emph{Relations and Kleene
  Algebra in Computer Science. RelMiCS 2006},  Lecture Notes in Computer
  Science  \textbf{4136} (2006), pp. 162--176.

\bibitem{Dzik2006b}
Dzik, W., E.~Orlowska and C.~van Alten, \emph{Relational representation
  theorems for lattices with negations: {A} survey}, Lecture Notes in
  Artificial Intelligence \textbf{4342} (2006), pp.~245--266.

\bibitem{FischerServi1977}
Fischer~Servi, G., \emph{On modal logic with an intuitionistic base}, Studia
  Logica \textbf{36} (1977), pp.~141--149.

\bibitem{Fitch1952}
Fitch, F.~B., \enquote{Symbolic Logic: An Introduction,} The Ronald Press
  Company, New York, 1952.

\bibitem{Fitch1966}
Fitch, F.~B., \emph{Natural deduction rules for obligation}, American
  Philosophical Quarterly \textbf{3} (1966), pp.~27--38.

\bibitem{Gehrke2006}
Gehrke, M., \emph{Generalized {K}ripke frames}, Studia Logica \textbf{84}
  (2006), pp.~241--275.

\bibitem{Gehrke2001}
Gehrke, M. and J.~Harding, \emph{Bounded lattice expansions}, Journal of
  Algebra \textbf{238} (2001), pp.~345--371.

\bibitem{Gehrke2005}
Gehrke, M., J.~Harding and Y.~Venema, \emph{Mac{N}eille completions and
  canonical extensions}, Transactions of the American Mathematical Society
  \textbf{358} (2005), pp.~573--590.

\bibitem{Gentzen1935}
Gentzen, G., \emph{Untersuchungen \"{u}ber das logische {S}chlie{\ss}en},
  Mathematische Zeitschrift \textbf{39} (1935), pp.~176--210, 405--431.

\bibitem{Gentzen1936}
Gentzen, G., \emph{{Die Widerspruchsfreiheit der reinen Zahlentheorie}},
  Mathematische Annalen \textbf{112} (1936), pp.~493--565.

\bibitem{Godel1933}
G\"{o}del, K., \emph{Zur intuitionistischen {A}rithmetik und {Z}ahlentheorie},
  Ergebnisse eines Mathematischen Kolloquiums \textbf{4} (1933), pp.~34--38.

\bibitem{Goldblatt2019}
Goldblatt, R., \emph{Morphisms and duality for polarities and lattices with
  operators} (2019), arXiv:1902.09783 [math.LO].

\bibitem{Goldblatt1974}
Goldblatt, R.~I., \emph{Semantic analysis of orthologic}, Journal of
  Philosophical Logic \textbf{3} (1974), pp.~19--35.

\bibitem{Hartonas2018}
Hartonas, C., \emph{Discrete duality for lattices with modal operators},
  Journal of Logic and Computation \textbf{29} (2018), pp.~71--89.

\bibitem{Hartonas2019}
Hartonas, C. and E.~Or\l{}owska, \emph{Representation of lattices with modal
  operators in two-sorted frames}, Fundamenta Informaticae \textbf{166} (2019),
  pp.~29--56.

\bibitem{Heyting1930}
Heyting, A., \emph{Die formalen {R}egeln der intuitionistischen {L}ogik {I}},
  Sitzungsberichte der Preussischen Akademie der Wissenschaften \textbf{49}
  (1930), pp.~42--65.

\bibitem{Holliday2015}
Holliday, W.~H., \emph{Possibility frames and forcing for modal logic} (2015),
  forthcoming in \textit{The Australasian Journal of Logic}, {U}C Berkeley
  Working Paper in Logic and the Methodology of Science,
  https://escholarship.org/uc/item/0tm6b30q.

\bibitem{Holliday2021b}
Holliday, W.~H., \emph{Possibility semantics}, in: M.~Fitting, editor,
  \emph{Selected Topics from Contemporary Logics}, Landscapes in Logic, College
  Publications, 2021 pp. 363--476,
  \href{https://arxiv.org/abs/2405.06852}{arXiv:2405.06852 [math.LO]}.

\bibitem{Holliday2021}
Holliday, W.~H., \emph{Three roads to complete lattices: {O}rders,
  compatibility, polarity}, Algebra Universalis \textbf{82} (2021), article
  number 26.

\bibitem{Holliday2022}
Holliday, W.~H., \emph{Compatibility and accessibility: lattice representations
  for semantics of non-classical and modal logics}, in: D.~F. Duque and
  A.~Palmigiano, editors, \emph{Advances in Modal Logic, Vol.~14}, College
  Publications, London, 2022 pp. 507--529,
  \href{https://arxiv.org/abs/2201.07098}{arXiv:2201.07098 [math.LO]}.

\bibitem{Holliday2023}
Holliday, W.~H., \emph{A fundamental non-classical logic}, Logics \textbf{1}
  (2023), pp.~36--79.

\bibitem{Holliday2024}
Holliday, W.~H., \emph{Preconditionals}, in: I.~Sedl\'{a}r, editor, \emph{The
  Logica Yearbook 2023}, College Publications, Forthcoming
  \href{https://arxiv.org/abs/2402.02296}{arXiv:2402.02296 [math.LO]}.

\bibitem{Holliday-Mandelkern2022}
Holliday, W.~H. and M.~Mandelkern, \emph{The orthologic of epistemic modals},
  Journal of Philosophical Logic  (Forthcoming),
  \href{https://arxiv.org/abs/2203.02872}{arXiv:2203.02872 [cs.LO]}.

\bibitem{Jonsson1952a}
J\'{o}nsson, B. and A.~Tarski, \emph{{Boolean Algebras with Operators. Part
  I.}}, American Journal of Mathematics \textbf{73} (1951), pp.~891--939.

\bibitem{Kenny1976}
Kenny, A., \emph{Human abilities and dynamic modalities}, in: J.~Manninen and
  R.~Tuomela, editors, \emph{Essays on Explanation and Understanding},
  Synthese Library  \textbf{72}, Springer, Dordrecht, 1976 pp. 209--232.

\bibitem{Massas2023}
Massas, G., \emph{B-frame duality}, Annals of Pure and Applied Logic
  \textbf{174} (2023), p.~103245.

\bibitem{Massas2024}
Massas, G., \emph{Goldblatt-thomason theorems for fundamental (modal) logic},
  in: A.~Ciabattoni and D.~Gabelaia, editors, \emph{Advances in Modal Logic,
  Vol.~14}, College Publications, London, Forthcoming ArXiv:2406.10182
  [math.LO].

\bibitem{Monting1981}
M\"{o}nting, J.~S., \emph{Cut elimination and word problems for varieties of
  lattices}, Algebra Universalis \textbf{12} (1981), pp.~290--321.

\bibitem{Orlowska2005}
Or\l{}owska, E. and D.~Vakarelov, \emph{Lattice-based modal algebras and modal
  logics}, in: P.~H\'{a}jek, L.~Vald\'{e}s-Villanueva and D.~Westerst{\aa}hl,
  editors, \emph{Logic, methodology and philosophy of science. Proceedings of
  the 12th international congress}, College Publications, London, 2005 pp.
  147--170.

\bibitem{Ploscica1995}
Plo\v{s}\v{c}ica, M., \emph{A natural representation of bounded lattices},
  Tatra Mountains Mathematical Publication \textbf{5} (1995), pp.~75--88.

\bibitem{Prenosil2023}
P{\v{r}}enosil, A., \emph{Compatibility between modal operators in distributive
  modal logic} (2023), arXiv:2311.10017 [math.LO].

\bibitem{Rebagliato1993}
Rebagliato, J. and V.~Verd\'{u}, \emph{On the algebraization of some {G}entzen
  systems}, Fundamenta Informaticae \textbf{17} (1993), pp.~319--338.

\bibitem{Urquhart1978}
Urquhart, A., \emph{A topological representation theory for lattices}, Algebra
  Universalis \textbf{8} (1978), pp.~45--58.

\bibitem{Vakarelov1989}
Vakarelov, D., \emph{Consistency, completeness and negation}, in: G.~Priest,
  R.~Routley and J.~Norman, editors, \emph{Paraconsistent Logic: {E}ssays on
  the Inconsistent}, Philosophia Verlag, Munich, 1989 pp. 328--368.

\bibitem{Wijesekera1990}
Wijesekera, D., \emph{Constructive modal logics {I}}, Annals of Pure and
  Applied Logic \textbf{50} (1990), pp.~271--301.

\bibitem{Zhong2021}
Zhong, S., \emph{A general relational semantics of propositional logic:
  Axiomatization}, in: A.~Silva, R.~Wassermann and R.~Queiroz, editors,
  \emph{Logic, Language, Information, and Computation. WoLLIC 2021},  Lecture
  Notes in Computer Science  \textbf{13038} (2021), pp. 82--99.

\end{thebibliography}

\end{document}